\documentclass{article}

\usepackage[
  lang=american,
]{ems-journal}
\usepackage{mathtools}
\usepackage{microtype}

\newcommand{\Z}{\mathbb Z}
\newcommand{\R}{\mathbb R}
\newcommand{\C}{\mathbb C}
\newcommand{\T}{\mathbb T}

\newcommand{\dd}{\mathop{}\!\mathrm{d}}
\newcommand{\one}{\mathbf{1}}

\DeclareMathOperator{\Id}{Id}

\theoremstyle{plain}
\newtheorem{theorem}{Theorem}
\newtheorem{proposition}[theorem]{Proposition}
\newtheorem{lemma}[theorem]{Lemma}
\newtheorem{corollary}[theorem]{Corollary}
\newtheorem{remark}[theorem]{Remark}
\newtheorem{definition}[theorem]{Definition}

\numberwithin{equation}{section}

\begin{document}

\title{Dimension-free estimates for full discrete maximal functions associated with Euclidean balls and spheres}
\titlemark{Dimension-free estimates for discrete maximal functions of balls and spheres}

\emsauthor{1}{
  \givenname{Mahdi}
  \surname{Hormozi}}{M.~Hormozi}

\Emsaffil{1}{
  \organisation{Beijing Institute of Mathematical Sciences and Applications}
  \zip{101408}
  \city{Beijing}
  \country{China}
  \affemail{hormozi@bimsa.cn}}

\emsauthor{2}{
  \givenname{Jakub}
  \surname{Niksi{\'n}ski}}{J.~Niksi{\'n}ski}

\Emsaffil{2}{
  \department{Department of Mathematics}
  \organisation{Rutgers University}
  \address{Hill Center -- Busch Campus, 110 Frelinghuysen Road}
  \zip{08854-8019}
  \city{Piscataway, NJ}
  \country{USA}
  \affemail{jakub.niksinski.math@gmail.com}}

\emsauthor{3}{
  \givenname{B{\l}a{\.z}ej}
  \surname{Wr{\'o}bel}}{B.~Wr{\'o}bel}

\Emsaffil{3}{
  \department{1}{}
  \organisation{1}{Institute of Mathematics of the Polish Academy of Sciences}
  \address{1}{{\'S}niadeckich 8}
  \zip{1}{00-656}
  \city{1}{Warsaw}
  \country{1}{Poland}

  \department{2}{Institute of Mathematics}
  \organisation{2}{University of Wroc{\l}aw}
  \address{2}{Plac Grunwaldzki 2}
  \zip{2}{50-384}
  \city{2}{Wroc{\l}aw}
  \country{2}{Poland}
  \affemail{2}{bwrobel@impan.pl}}


\hypersetup{
  pdftitle={Dimension-free estimates for full discrete maximal functions associated with Euclidean balls and spheres},
  pdfauthor={Mahdi Hormozi, Jakub Niksinski, Blazej Wrobel},
  pdfsubject={Dimension-free estimates for full discrete maximal functions associated with Euclidean balls and spheres},
  pdfkeywords={dimension-free estimates, discrete Hardy--Littlewood maximal function,
Euclidean balls, Euclidean spheres}
}

\begin{abstract}
We prove that the full discrete Hardy--Littlewood maximal operator associated with Euclidean balls satisfies dimension-free bounds on $\ell^p(\mathbb Z^d)$ for every $1<p<\infty$. We also establish analogous dimension-free bounds for the full discrete spherical maximal operator when $d\geq 5$ and $2\leq p<\infty$. {The main new idea is to approximate the relevant Fourier multipliers by finite linear combinations of normalized discrete Gaussian multipliers and their translates. We obtain these approximations uniformly in frequency and with uniformly bounded coefficients through a refined saddle-point analysis.} The ball result resolves a question of E.~M.~Stein, while the spherical result gives, in the range $p\geq 2$, a dimension-free strengthening of the theorem of Magyar, Stein, and Wainger.

\end{abstract}

\makeatletter
\let\ps@titlepage\ps@empty
\makeatother
\maketitle

\section{Introduction}

\subsection{Statement of the results}

For an integer \(d\geq 1\) and \(t\geq 0\), set
\begin{equation*}
%\label{eq:euclidean-ball-definition}
 \mathcal B_t^d
 =\left\{x\in\mathbb Z^d:\ |x|^2=x_1^2+\cdots+x_d^2\leq t\right\},
\end{equation*}
that is, \(\mathcal B_t^d\) is the set of lattice points in $\Z^d$ which lie inside the closed Euclidean ball of radius $\sqrt{t}$.
For \(f:\mathbb Z^d\to\mathbb C\) and \(x\in\mathbb Z^d\), define the Hardy--Littlewood averaging operator over these discrete balls
\begin{equation}\label{eq:euclidean-averages}
 \mathcal M_t^df(x)
 =\frac1{|\mathcal B_t^d|}\sum_{y\in\mathcal B_t^d}f(x-y)
\end{equation}
and the corresponding Hardy--Littlewood maximal function
\begin{equation*}
 \mathcal M_*^df(x)=\sup_{t\geq 0}|\mathcal M_t^df(x)|.
\end{equation*}
Since $\mathcal B_t^d=\mathcal B_n^d$ whenever $n=\lfloor t\rfloor$, we have
\begin{equation*}
%\label{eq:euclidean-maximal}
 \mathcal M_*^df(x)=\sup_{n\geq 0}|\mathcal M_n^df(x)|.
\end{equation*}
We prove that $\mathcal M_*^d$ has dimension-free bounds on all $\ell^p(\Z^d)$ spaces with $1<p<\infty$.
\begin{theorem}\label{thm:main}
For every \(1<p\leq\infty\), there is a constant \(C_p<\infty\), independent
of \(d\), such that, for every \(d\geq 1\) and every
\(f\in\ell^p(\mathbb Z^d)\),
\[
 \|\mathcal M_*^df\|_{\ell^p(\mathbb Z^d)}
 \leq C_p\|f\|_{\ell^p(\mathbb Z^d)}.
\]
\end{theorem}
\noindent Our result answers in the affirmative a question of E.~M.~Stein \cite{SteinPC}, who asked whether
\begin{equation}
\label{eq:SQ}
 \sup_{d\geq 1}
 \|\mathcal M_*^d\|_{\ell^2(\mathbb Z^d)\to\ell^2(\mathbb Z^d)}<\infty;
\end{equation}
see \cite[p.~903]{MirekSzarekWrobelSpheres} and \cite[p.~3]{MirekSzarekWrobelGaussians}. Theorem~\ref{thm:main} not only answers
this question but gives a dimension-free bound for every
\(1<p\leq\infty\). It is a natural discrete counterpart of a celebrated result of E.~M.~Stein for continuous balls on $\R^d$ \cite[Theorem~13, p.~374]{Stein1982} (with a detailed proof given by Stein and Str\"omberg \cite[Appendix, pp.~266--269]{SteinStromberg1983}), which is one of the cornerstones of high-dimensional harmonic analysis. Actually, Theorem \ref{thm:main} {also implies} Stein's continuous result. This is because the optimal constant in the corresponding continuous
maximal inequality is always {bounded above} by its discrete counterpart;
see \cite[Theorem~1]{KoszMirekPlewaWrobel} for this comparison in the more general
setting of symmetric convex bodies.

We also prove a corresponding dimension-free estimate for the spherical maximal operator. For \(d\geq 5\) and an integer \(n\geq 0\), let
\begin{equation*}
%\label{eq:euclidean-sphere-definition}
 \mathcal S_n^d
 =\left\{x\in\mathbb Z^d:\ |x|^2=x_1^2+\cdots+x_d^2=n\right\},
\end{equation*}
that is, $\mathcal S_n^d$ is the set of lattice points in $\Z^d$ which lie on the Euclidean sphere of radius $\sqrt{n}$. For \(f:\mathbb Z^d\to\mathbb C\) and \(x\in\mathbb Z^d\), define the spherical averaging operator
\begin{equation}\label{eq:spherical-averages}
 \mathcal A_n^df(x)
 =\frac1{|\mathcal S_n^d|}\sum_{y\in\mathcal S_n^d}f(x-y)
\end{equation}
and the corresponding spherical maximal function
\begin{equation*}
%\label{eq:spherical-maximal}
 \mathcal A_*^df(x)=\sup_{n\geq 0}|\mathcal A_n^df(x)|.
\end{equation*}
We prove that $\mathcal A_*^d$ has dimension-free bounds on all $\ell^p(\Z^d)$ spaces with $2\leq p<\infty$.
\begin{theorem}\label{thm:main:sphere}
For every \(2\leq p\leq\infty\), there is a constant \(C_p<\infty\), independent
of \(d\), such that, for every \(d\geq 5\) and every
\(f\in\ell^p(\mathbb Z^d)\),
\[
 \|\mathcal A_*^df\|_{\ell^p(\mathbb Z^d)}
 \leq C_p\|f\|_{\ell^p(\mathbb Z^d)}.
\]
\end{theorem}
\noindent This theorem is a dimension-free strengthening of a result of A.~Magyar, E.~M.~Stein, and S.~Wainger \cite{MagyarSteinWainger} (see also Ionescu \cite{Ionescu} for an endpoint estimate). Since the averages over discrete balls in \eqref{eq:euclidean-averages} are convex combinations of the averages over discrete spheres in \eqref{eq:spherical-averages}, we have the pointwise inequality $\mathcal M_*^df(x)\leq \mathcal A_*^d|f|(x)$. Therefore, Theorem~\ref{thm:main:sphere} is stronger than Theorem~\ref{thm:main} for $p\geq 2$. 
The restriction \(p\geq2\) in Theorem~\ref{thm:main:sphere} comes from the large-scale estimate \eqref{eq:max sphere large scale l2}. Indeed, the presently available dimension-free estimate for the full spherical maximal function at large radii is an \(\ell^2\) estimate, and interpolation with the trivial \(\ell^\infty\) bound yields \eqref{eq:max sphere large scale} only for \(p\geq2\). Note that the Gaussian approximation we use to treat the range \(n\leq d^4\) does not itself impose the restriction $p\ge 2$. We do not address here whether the dimension-free estimate in Theorem \ref{thm:main:sphere} extends below \(p=2\).

\subsection{A brief history}

We now briefly describe existing results, focusing on those which are most relevant for our work. For a more detailed account we refer the reader to \cite[Section~1.2]{MirekSzarekWrobelGaussians} or \cite[Section~1.2]{NiksinskiWrobel}.

\subsection{Discrete Euclidean balls and Theorem~\ref{thm:main}}

First, we recall that there is no dimension-free discrete estimate uniform over all symmetric
convex bodies. Namely, in \cite[Theorem~2]{BMSWcubes}, J.~Bourgain, M.~Mirek, E.~M.~Stein, and the third author constructed a sequence
of ellipsoids for which the $\ell^p(\Z^d)$-operator norm, $1<p<\infty$, of the full discrete maximal function is
at least \(c_p(\log d)^{1/p}\), where \(c_p>0\) depends only on $p$. However, this example does
not rule out a dimension-free estimate for Euclidean balls.

By interpolating a trivial $\ell^{\infty}(\Z^d)$ bound with a weak-type $(1,1)$ estimate, one easily sees that \(\|\mathcal M_*^d\|_{\ell^p(\mathbb Z^d)\to\ell^p(\mathbb Z^d)}\) is finite for every $d\geq 1$ and $p\in(1,\infty)$. However, classical methods yield constants which grow exponentially with the dimension. The first results towards Stein's question \eqref{eq:SQ} were proved by J.~Bourgain, M.~Mirek, E.~M.~Stein, and the third author \cite{BMSWcubes,BMSWballs,BMSWperspective}. Namely, using a comparison principle from \cite[Theorem~1]{BMSWcubes}, which was sharpened in \cite[Theorem~2]{BMSWperspective}, they deduced that, for each $1<p<\infty$,
\begin{equation}
\label{eq:max ball large scale}
 \left\|\sup_{n>Cd^2}|\mathcal M_n^df|\right\|_{\ell^p(\mathbb Z^d)}\leq C_p\|f\|_{\ell^p(\mathbb Z^d)};
\end{equation}
here $C>0$ is a universal constant and $C_p$ depends only on $p$. This was achieved by comparing the discrete Hardy--Littlewood maximal operator with its continuous counterpart, for which Stein's estimate \cite[Theorem~13, p.~374]{Stein1982} applies. 

The remaining task was to treat the supremum over smaller scales $n<Cd^2$. However, it was not clear what the appropriate objects for comparison at these scales would be. This difficulty was partially overcome in \cite{BMSWballs} and \cite{NiksinskiWrobel}. Namely, in \cite[Theorem~1.1]{BMSWballs}, the authors proved a dyadic counterpart of \eqref{eq:SQ}. Recently, the second and third authors of the present paper \cite[Theorem~1.1]{NiksinskiWrobel} proved a weaker variant of \eqref{eq:SQ}, with the supremum restricted to small radii $n<d^{1-\varepsilon},$ where $0<\varepsilon<1$ is arbitrarily small. We note that in both \cite{NiksinskiWrobel} and \cite{BMSWballs} (at small scales $n\leq cd$), the analysis was essentially reduced to product settings modeled on the Hamming cube, such as $\{-1,0,1\}^d$ or, more generally, $\{-K,-K+1,\ldots,0,1,\ldots,K\}^d$, and the objects used for comparison were various combinatorial multipliers defined via Krawtchouk polynomials. Even more recently the second author in \cite{Niksinski} generalized these works (\cite{BMSWballs}, \cite{NiksinskiWrobel}) to the setting of discrete Hardy-Littlewood averaging operators over 1-symmetric convex bodies, i.e. those that are invariant under sign changes and permutations of coordinates. Due to these partial results he conjectured an analogue of Stein's question for 1-symmetric convex bodies. A resolution of this conjecture would give a discrete analogue of Bourgain's celebrated result \cite{Bourgain}.  It seems unlikely that the methods of our work, without significant new ideas, can be used towards attacking the aforementioned conjecture from \cite{Niksinski}. Nevertheless it is a direction that we plan to pursue in the future.

One of the main ideas of our paper is that the appropriate objects for comparison are in fact linear combinations, with uniformly bounded coefficients, of normalized discrete Gaussians. From this perspective the most important previous work is \cite{MirekSzarekWrobelGaussians}. There, Mirek, Szarek, and the third author proved dimension-free bounds for discrete maximal functions related to normalized discretely sampled Gaussians.
For \(|z|<1\), define
\[
 \Theta(z)=\sum_{k\in\mathbb Z}z^{k^2},
\]
and, for \(d\geq 1\) and \(0<r<1\), set
\[
 g_r^d(x)=\Theta(r)^{-d}r^{|x|^2},
 \qquad x\in\mathbb Z^d.
\]
Consider the following averaging operators and maximal functions:
\[
 \mathcal G_r^df=g_r^d*f,
 \qquad
 \mathcal G_*^df=\sup_{0<r<1}|\mathcal G_r^df|.
\]
Under the reparameterization $r=\exp(-\pi/t)$, the operators $\mathcal G_r^d$ and $\mathcal G_*^d$ coincide with $G_t$ and $G_*$ considered in \cite{MirekSzarekWrobelGaussians}. One of the main results of \cite{MirekSzarekWrobelGaussians} is thus a dimension-free $\ell^p(\Z^d)$ bound for $G_*$, which we restate for further reference.
\begin{theorem}[{After \cite[Theorem~1.1, (1.5)]{MirekSzarekWrobelGaussians}}]\label{thm:gaussian}
For every \(1<p\leq\infty\), there is a constant \(C_p<\infty\), independent
of \(d\), such that, for every \(d\geq 1\) and every
\(f\in\ell^p(\mathbb Z^d)\),
\[
 \|\mathcal G_*^df\|_{\ell^p(\mathbb Z^d)}
 \leq C_p\|f\|_{\ell^p(\mathbb Z^d)}.
\]
\end{theorem}
\noindent Theorem~1.1 of \cite{MirekSzarekWrobelGaussians} also contains dimension-free
jump and variational inequalities with variation
exponent greater than two.

\subsection{Discrete spheres and Theorem~\ref{thm:main:sphere}}

In contrast with the discrete ball case, it is not at all clear a priori that $\mathcal A_*^d$ is bounded on $\ell^p(\Z^d)$, let alone whether it satisfies dimension-free estimates. Its boundedness properties were determined by Magyar, Stein, and Wainger \cite{MagyarSteinWainger}, who showed that, in dimensions $d\geq 5$, it is indeed bounded on $\ell^p(\Z^d)$ if and only if $p>d/(d-2)$. See also Magyar \cite{Magyar} for a previous result of this kind and Ionescu \cite{Ionescu} for an endpoint estimate. For $d=4$, due to the irregularity of the number of representations of a natural number as a sum of four squares, the corresponding maximal operator is unbounded for every $p<\infty$. If one restricts the supremum to odd $n$'s, then the $\ell^p$ boundedness of the corresponding maximal operator is one of the major open problems in this field.

It should be noted that the comparison principles from \cite[Theorem~1]{BMSWcubes} and \cite[Theorem~2]{BMSWperspective}
rely on the nestedness of the balls $\mathcal B_{n_1}^d\subseteq\mathcal B_{n_2}^d$, $0<n_1<n_2$, which has no analogue for the Euclidean spheres $\mathcal S_n^d$. However, a substitute for \eqref{eq:max ball large scale} on $\ell^2(\Z^d)$ is still available for the spheres. Namely, an observation of M.~Mirek, T.~Z.~Szarek, and the third author \cite[Remark~5.3, (5.19)]{MirekSzarekWrobelSpheres} implies that
there exists a universal constant $C>0$ such that
\begin{equation}
\label{eq:max sphere large scale l2}
 \left\|\sup_{n>Cd^3}|\mathcal A_n^df|\right\|_{\ell^2(\mathbb Z^d)}\leq C\|f\|_{\ell^2(\mathbb Z^d)}.
\end{equation}
To establish this bound, the authors invoked an estimate for the derivative of the continuous spherical multiplier, namely \cite[Remark~5.3, (5.20)]{MirekSzarekWrobelSpheres}, yet without including a proof. Inequality \cite[Remark~5.3, (5.20)]{MirekSzarekWrobelSpheres} was later justified in detail by Ciccone and the third author, see \cite[Proposition~5.3]{CicconeWrobel}.
By interpolating \eqref{eq:max sphere large scale l2} with the trivial $\ell^{\infty}(\Z^d)$ bound, we conclude that, for every $2\leq p\leq\infty$,
\begin{equation}
\label{eq:max sphere large scale}
 \left\|\sup_{n>Cd^3}|\mathcal A_n^df|\right\|_{\ell^p(\mathbb Z^d)}\leq C_p\|f\|_{\ell^p(\mathbb Z^d)};
\end{equation}
here $C>0$ is a universal constant and $C_p$ depends only on $p$.

\subsection{Our methods and the structure of the paper}

Let
\[
 m_n^d(\xi)
 =\frac1{|\mathcal B_n^d|}
 \sum_{x\in\mathcal B_n^d}e^{2\pi ix\cdot\xi},\qquad \xi\in\mathbb T^d,
\]
be the multiplier symbol corresponding to the average $\mathcal M_n^d$, where $\mathbb T^d=\mathbb R^d/\mathbb Z^d$, and set
\[
 \mathbf h=\left(\frac12,\ldots,\frac12\right)\in\mathbb T^d.
\]
The next proposition is the main new tool in our analysis for the ball case. It says that, for $1\leq n\leq d^4$,
each ball multiplier is, up to a remainder
\(O_{N}(\min\{n,d\}^{-N})\), a finite linear combination, with uniformly bounded coefficients, of normalized
discrete Gaussian multipliers and their translates by
\(\mathbf h\).
\begin{proposition}\label{prop:euclidean-Gaussian-approximation}
Fix an integer dimension $d\geq 1$ and an integer squared radius $1\leq n\leq d^4$. Then, for every integer \(N\geq 1\), there exist
\begin{enumerate}
\item an integer \(L_N\geq 1\) and a constant $C_N>0$ depending only on $N$;
\item complex coefficients
\(\alpha_{\nu}=\alpha_{\nu,d,n}\) and \(\beta_{\nu}=\beta_{\nu,d,n}\), for \(1\leq\nu\leq L_N\);
\item parameters
\(r_{\nu}=r_{\nu,d,n}\) and \(\widetilde r_{\nu}=\widetilde r_{\nu,d,n}\) in $(0,1)$, for \(1\leq\nu\leq L_N\),
\end{enumerate}
such that
\begin{equation}\label{eq:euclidean-Gaussian-coefficients}
 \sum_{\nu=1}^{L_N}
 \bigl(|\alpha_{\nu}|+|\beta_{\nu}|\bigr)
 \leq C_{N},
\end{equation}
and
\begin{equation}\label{eq:euclidean-multiplier-remainder}
 \sup_{\xi\in\mathbb T^d}\left|m_n^d(\xi)-\sum_{\nu=1}^{L_N}\alpha_{\nu}\widehat{g_{r_{\nu}}^d}(\xi)
 -\sum_{\nu=1}^{L_N}\beta_{\nu}
 \widehat{g_{\widetilde r_{\nu}}^d}(\xi-\mathbf h)\right|
 \leq C_{N}\min\{n,d\}^{-N}.
\end{equation}
\end{proposition}

For the spherical case, our main tool will be a variant of Proposition~\ref{prop:euclidean-Gaussian-approximation}. Here the corresponding spherical multiplier is given by
\[
 s_n^d(\xi)
 =\frac1{|\mathcal S_n^d|}
 \sum_{x\in\mathcal S_n^d}e^{2\pi ix\cdot\xi},\qquad \xi\in\mathbb T^d.
\]

\begin{proposition}\label{prop:sphere-Gaussian-approximation}
Fix a dimension $d\geq 5$ and a squared radius $1\leq n\leq d^4$. Then, for every integer \(N\geq 1\), there exist
\begin{enumerate}
\item an integer \(L_N\geq 1\) and a constant $C_N>0$ depending only on $N$;
\item complex coefficients
\(\alpha_{\nu}=\alpha_{\nu,d,n}\) and \(\beta_{\nu}=\beta_{\nu,d,n}\), for \(1\leq\nu\leq L_N\);
\item parameters
\(r_{\nu}=r_{\nu,d,n}\) and \(\widetilde r_{\nu}=\widetilde r_{\nu,d,n}\) in $(0,1)$, for \(1\leq\nu\leq L_N\),
\end{enumerate}
such that
\begin{equation}\label{eq:sphere-Gaussian-coefficients}
 \sum_{\nu=1}^{L_N}
 \bigl(|\alpha_{\nu}|+|\beta_{\nu}|\bigr)
 \leq C_{N},
\end{equation}
and
\begin{equation}\label{eq:sphere-multiplier-remainder}
 \sup_{\xi\in\mathbb T^d}\left|s_n^d(\xi)-\sum_{\nu=1}^{L_N}\alpha_{\nu}\widehat{g_{r_{\nu}}^d}(\xi)
 -\sum_{\nu=1}^{L_N}\beta_{\nu}
 \widehat{g_{\widetilde r_{\nu}}^d}(\xi-\mathbf h)\right|
 \leq C_{N}\min\{n,d\}^{-N}.
\end{equation}
\end{proposition}
Before moving further with the discussion of our methods we make some remarks on Propositions \ref{prop:euclidean-Gaussian-approximation} and \ref{prop:sphere-Gaussian-approximation}. Below for
$\theta \in \R$ and $|z|<1$ we define
\[
H_\theta(z)=\sum_{k\in\mathbb Z}z^{k^2}e^{2\pi i k\theta}.
\]
Note that then \(H_0=\Theta\). Properties of the functions $H_{\theta}$ will be crucial in our analysis. This is because $H_{0}$ appears in the generating-function representation of the number of lattice points in the ball $|B_n^d|,$ while $H_{\theta}$ appears in the generating-function representation of $|B_n^d|m_n^d(\xi).$

\begin{remark}
   The appearance of Gaussian multipliers translated by
$\mathbf h$
is natural for two related reasons. Firstly, the parity identity $k^2\equiv k \pmod 2$
implies
\[
H_{\theta+1/2}(z)=H_\theta(-z),
\]
which relates the contributions to the coefficient integral coming from
neighborhoods of $t=0$ and $t=1/2$ in the formula \eqref{eq: Formula for multiplier}.
Secondly, the original multipliers themselves may be large near $\mathbf h$, whereas Gaussian multiplier at $\mathbf h$ should oscillate and be small.
Indeed,
\[
s_n^d(\mathbf h)=(-1)^n,
\]
while
\[
m_n^d(\mathbf h)
=
\frac{\displaystyle\sum_{k=0}^n(-1)^k|\mathcal S_k^d|}
{\displaystyle\sum_{k=0}^n|\mathcal S_k^d|}.
\]
The latter quantity is close to $(-1)^n$ whenever the outer sphere
$\mathcal S_n^d$ dominates the ball $\mathcal B_n^d$, as happens, for
instance, for fixed $n$ and large $d$ (cf.\ \cite{NiksinskiWrobel}). Thus one should expect the
approximation to contain components localized both near $0$ and near
$\mathbf h$. 
\end{remark}
\begin{remark}
For every fixed $K>0$, both Propositions~\ref{prop:euclidean-Gaussian-approximation} and \ref{prop:sphere-Gaussian-approximation} remain true with $d^4$ replaced by $d^K$, with the constants allowed to depend on $K$. We focus on the case $K=4$ because it is enough for our purposes. Indeed, there are only finitely many dimensions for which $Cd^3>d^4$ and these dimensions are harmless, since the corresponding fixed-dimensional bounds can be absorbed into the constants. For all remaining dimensions, the two propositions, together with the large-scale estimates \eqref{eq:max ball large scale} and \eqref{eq:max sphere large scale}, cover all squared radii.
\end{remark}
\begin{remark}
The constants, coefficients, and parameters in Propositions~\ref{prop:euclidean-Gaussian-approximation} and \ref{prop:sphere-Gaussian-approximation} are not the same. What matters is that they satisfy \eqref{eq:euclidean-Gaussian-coefficients}, \eqref{eq:euclidean-multiplier-remainder}, \eqref{eq:sphere-Gaussian-coefficients}, and \eqref{eq:sphere-multiplier-remainder}.
\end{remark}

Propositions~\ref{prop:euclidean-Gaussian-approximation} and
\ref{prop:sphere-Gaussian-approximation}, together with
Theorem~\ref{thm:gaussian}, allow us to deduce
Theorems~\ref{thm:main} and \ref{thm:main:sphere}. Indeed, the Gaussian
parts of the approximating multipliers are controlled by
Theorem~\ref{thm:gaussian} (here the translates by $\mathbf h$ are handled
by modulation). The remainders are first controlled on $\ell^2(\Z^d)$
using Parseval's identity and then on $\ell^p(\Z^d)$ by interpolation
with trivial $\ell^1(\Z^d)$ or $\ell^\infty(\Z^d)$ bounds. Choosing the
order of approximation sufficiently large and summing over
$1\leq n\leq d^4$ gives the required maximal estimates. We refer to
Section~\ref{ssec: main from prop} for details.

We finish this section with a discussion of the proofs of
Propositions~\ref{prop:euclidean-Gaussian-approximation} and
\ref{prop:sphere-Gaussian-approximation}. Our aim here is to give the reader a high-level yet fairly complete picture of the main ideas in these proofs. Therefore some repetitions are unavoidable later in Sections \ref{sec: balls} and \ref{sec: spheres}. 

We discuss the ball case first. Our starting points are the aforementioned generating-function representations
\[
\sum_{n=0}^{\infty}|\mathcal B_n^d|z^n
=
\frac{\Theta(z)^d}{1-z},
\qquad
\sum_{n=0}^{\infty}
|\mathcal B_n^d|m_n^d(\xi)z^n
=
\frac{\prod_{j=1}^dH_{\xi_j}(z)}{1-z}.
\]
Thus, by Cauchy's formula, for every \(r\in(0,1)\),
\[
|\mathcal B_n^d|
=
\frac{1}{2\pi i}
\oint_{|z|=r}
\frac{\Theta(z)^d}{(1-z)z^{n+1}}\dd z,
\qquad
|\mathcal B_n^d|m_n^d(\xi)
=
\frac{1}{2\pi i}
\oint_{|z|=r}
\frac{\prod_{j=1}^dH_{\xi_j}(z)}
     {(1-z)z^{n+1}}\dd z.
\]

To set up the saddle-point method, we choose the radius
\(r=r_{d,n}\in(0,1)\) according to
\[
\left.
\frac{d}{dx}
\log\left(
\frac{\Theta(x)^d}{(1-x)x^n}
\right)
\right|_{x=r}
=0.
\]
Writing
\[
\rho=1-r,
\qquad
\mu(r)=\frac{r\Theta'(r)}{\Theta(r)},
\]
this condition becomes
\[
n=\frac{r}{\rho}+d\mu(r).
\]

Having fixed \(r\), we parameterize the contour by
\[
z=re^{2\pi it},
\qquad -\frac12\leq t\leq\frac12,
\]
and set
\[
\Phi_{d,n}(t)
=
e^{-2\pi int}
\frac{\rho}{1-re^{2\pi it}}
\left(
\frac{\Theta(re^{2\pi it})}{\Theta(r)}
\right)^d.
\]
Cauchy's formula then takes the form
\[
I_{d,n}
:=
\rho\Theta(r)^{-d}r^n|\mathcal B_n^d|
=
\int_{-1/2}^{1/2}\Phi_{d,n}(t)\dd t.
\]
By the choice of \(r\), the linear term in the Taylor expansion of
\(\log\Phi_{d,n}\) vanishes at $0$. More precisely, for
\(|t|\leq c\rho\) we 
\[
\log\Phi_{d,n}(t)
=
-2\pi^2Vt^2
+
O\left(\frac{V}{\rho}|t|^3\right),
\]
where
\[
V
=
\frac{r}{\rho^2}
+
dr\mu'(r)
\asymp
\frac{dr}{\rho^2}.
\]
Consequently,  the coefficient integral is concentrated on the scale
\(V^{-1/2}\), while Taylor's expansion is valid up to the larger scale
\(c\rho\).

The appearance of the latter scale is motivated by
Corollary~\ref{corollary: psi decay}. Indeed, it gives quantitative decay of
\[
\frac{H_\theta(re^{2\pi it})}{H_0(r)}
\]
in terms of the distance of \((t,\theta)\) from the two exceptional
points \((0,0)\) and \((1/2,1/2)\). These are the only points of
\(\R^2/\Z^2\) at which the quotient has modulus one. In particular,
when \(\theta=0\) and \(|t|\geq c\rho\), Corollary~\ref{corollary: psi decay} gives
\[
\frac{|H_0(re^{2\pi it})|}{H_0(r)}
\leq e^{-cr}
\]
and taking the \(d\)-th power, we obtain
\[
\left(
\frac{|H_0(re^{2\pi it})|}{H_0(r)}
\right)^d
\leq e^{-crd}
\leq e^{-c'\min\{n,d\}}.
\]
Thus the possibly weak decay of a single theta factor is amplified by
the product over the \(d\) coordinates.

The scale $\rho$ also has a complementary heuristic interpretation. Indeed,
\[
H_0(r)=\sum_{k\in\Z}r^{k^2}
\]
is essentially concentrated on \(|k|\lesssim\rho^{-1/2}\). Hence one
may think of \(H_0(re^{2\pi it})\) as a quadratic exponential sum over
this range, whose phase begins to oscillate when
\(|t|\gtrsim\rho\). In fact, the proof of
Corollary~\ref{corollary: psi decay} is based on a reduction to such
an exponential sum.

These two scales $\rho$ and $V^{-1/2}$ naturally lead to the parameter
\[
\kappa=\rho^2V\asymp\min\{n,d\}.
\]
In other words $\kappa$ is the square of the ratio between the scale \(\rho\), on which we
use Taylor's expansion, and the scale \(V^{-1/2}\), on which the
integral concentrates. Indeed, after the rescaling
\[
t=\frac{v}{\sqrt V},
\]
the local interval \(|t|\leq c\rho\) becomes
\[
|v|\leq c\sqrt\kappa,
\]
while the cubic error becomes
\[
O\left(\frac{|v|^3}{\sqrt\kappa}\right).
\]
Thus \(\kappa\) controls the accuracy of the saddle-point expansion.
More precisely, Lemma~\ref{lemma:decay of Phi}, whose proof uses
Corollary~\ref{corollary: psi decay}, gives Gaussian decay
\(e^{-cVt^2}\) on the local interval and exponential decay
\(e^{-c\kappa}\) on its complement. Combining these estimates with
the local Gaussian furnished by the saddle-point expansion approximation yields
\[
I_{d,n}\asymp V^{-1/2}.
\]

We remark in passing that the saddle-point analysis developed here to approximate $I_{d,n}$ has further applications to
lattice-point counting. In ongoing work with L.\ Daskalakis, we use it to
obtain sharp asymptotics, uniform in both the dimension and the radius, for
the numbers of lattice points in discrete Euclidean balls and spheres.

We now turn to the Cauchy integral representation of \(m_n^d(\xi)\).
The main additional feature is that two frequency regions have to be taken into
account, one near $0$ and the other near
\[
\mathbf h=\left(\frac12,\ldots,\frac12\right).
\]
To distinguish them, set
\[
\Lambda(\xi)
=
\sum_{j=1}^d\sin^2(\pi\xi_j).
\]
Since
\[
\Lambda(\xi-\mathbf h)
=
\sum_{j=1}^d\cos^2(\pi\xi_j)
=
d-\Lambda(\xi),
\]
for at least one of the arguments \(\xi\) or \(\xi-\mathbf h\) the corresponding value of
\(\Lambda\) does not exceed $d/2$. Therefore, for a given $\xi\in \T^d,$ we choose
\[
(\varepsilon,\eta)
=
\begin{cases}
(1,\xi),&\text{if }\Lambda(\xi)\leq d/2,\\
(-1,\xi-\mathbf h),&\text{if }\Lambda(\xi)>d/2.
\end{cases}
\]
Then \(\Lambda(\eta)\leq d/2\) and
\[
\xi=\eta+\frac{1-\varepsilon}{2}\mathbf h.
\]
Thus the parameter \(\varepsilon\) records whether we work directly with \(\xi\) or
first translate it by \(\mathbf h\).

The parity identity
\[
H_{\theta+1/2}(z)=H_\theta(-z)
\]
now allows us to treat the two cases simultaneously. Namely, for $\varepsilon\in\{-1,1\}$ define
\[
R_{\varepsilon,\eta}(t)
=
\frac{1-re^{2\pi it}}
     {1-\varepsilon re^{2\pi it}}
\prod_{j=1}^d
\frac{H_{\eta_j}(re^{2\pi it})}
     {H_0(re^{2\pi it})}.
\]
After a change of variables when \(\varepsilon=-1\), Cauchy's formula
gives
\[
\varepsilon^n m_n^d(\xi)
=
\frac{1}{I_{d,n}}
\int_{-1/2}^{1/2}
\Phi_{d,n}(t)R_{\varepsilon,\eta}(t)\dd t.
\]
In this way, the two frequency regions are reduced to the same
saddle-point problem.

With this representation in hand, we again split the integral at
\(|t|=c\rho\). On the local interval \(|t|\le c\rho\), we write the Taylor expansion of
\(R_{\varepsilon,\eta}\) at \(0\). To control the remainder uniformly
in \(d\), \(n\), and \(\eta\), it suffices to obtain uniform bounds for
the derivatives of
\[
\prod_{j=1}^d
\frac{H_{\eta_j}(re^{2\pi it})}
     {H_0(re^{2\pi it})}.
\]
The main new input is a bound for the quotient \(H_\theta(z)/H_0(z)\) when \(z\) is a
complex number close to the positive real axis. Using Jacobi's product
formula, we prove in Lemma~\ref{lemma: bound for quotient of H(z)}
that in the relevant region.
\[
\left|
\frac{H_\theta(z)}{H_0(z)}
\right|
\leq
\exp\left(
-c\frac{|z|}{1-|z|}
\sin^2(\pi\theta)
\right).
\]
Cauchy's estimates and the Leibniz formula
then imply the required derivative bounds, see
Corollary~\ref{corllary:euclidean-product-derivatives}.

On the complementary range \(|t|>c\rho\), we use
Corollary~\ref{corollary: psi decay} once more. Note that the condition
\(\Lambda(\eta)\leq d/2\) ensures that a fixed positive proportion of
the coordinates \(\eta_j\) stay uniformly away from the exceptional
frequency \(1/2\). Applying the corollary in these coordinates and
taking the product therefore gives
\[
|\Phi_{d,n}(t)R_{\varepsilon,\eta}(t)|
\leq e^{-c\kappa},
\qquad |t|>c\rho.
\]

Combining the local Taylor expansion with the tail estimate, we
obtain, for every fixed \(N\), an approximation of \(\varepsilon^n m_n^d(\xi)\) with error
\(O_N(\kappa^{-N})\). According to whether
\(\varepsilon=1\) or \(\varepsilon=-1\), the resulting terms are
derivatives at \(t=0\) of
\[
\prod_{j=1}^d
\frac{H_{\xi_j}(re^{2\pi it})}
     {H_0(re^{2\pi it})}
\]
or
\[
\frac{1-re^{2\pi it}}{1+re^{2\pi it}}
\prod_{j=1}^d
\frac{H_{\xi_j-1/2}(re^{2\pi it})}
     {H_0(re^{2\pi it})}.
\]
Corollary~\ref{corllary:euclidean-product-derivatives} also shows that
the terms associated with the other frequency region are
\(O_N(\kappa^{-N})\). We may therefore combine the two cases into a
single approximation. It is valid uniformly for all
\(\xi\in\mathbb T^d\) and involves derivatives associated with both
\(\xi\) and \(\xi-\mathbf h\).

It remains to convert these derivatives into genuine Gaussian
multipliers. Firstly, they can be expressed in terms of derivatives
with respect to \(r\) of normalized discrete Gaussian multipliers at
the frequencies \(\xi\) and \(\xi-\mathbf h\). Then after a convenient
reparameterization of the Gaussian family, we replace them by
sufficiently accurate high-order finite differences. Here the
discretization scale \(h\) in
\eqref{eq:euclidean-finite-difference} is chosen in terms of
\(\kappa\). This replacement follows from Taylor's theorem and
elementary linear algebra. At the same time, the decay of the
coefficients arising from the saddle-point expansion ensures that
the coefficients in the resulting finite linear combination remain
uniformly bounded.

Consequently, up to an error \(O_N(\kappa^{-N})\), the multiplier
\(m_n^d\) is a finite linear combination of genuine normalized
discrete Gaussian multipliers and their translates by \(\mathbf h\).
Since
\[
\kappa\asymp\min\{n,d\},
\]
this proves
Proposition~\ref{prop:euclidean-Gaussian-approximation}. The details
are given in Section~\ref{sec: balls}.

It is worth emphasizing that our use of the saddle-point method 
goes well beyond its earlier role in these problems. In \cite{MazoOdlyzko}, the method was
briefly suggested for lattice-point counting in the regime
\(n\asymp d\), while in \cite{NiksinskiWrobel} it was developed for counting
lattice points in balls and spheres only for small radii \(n\leq cd\). Neither work seems to have anticipated that the method
could yield uniform Gaussian approximations of the corresponding
multipliers throughout the range of \(n\le d^4\) and thus lead to the proofs of Theorems \ref{thm:main} and \ref{thm:main:sphere}.

The spherical case follows the same scheme. Here the generating
function is
\[
\sum_{n\geq0}|\mathcal S_n^d|s_n^d(\xi)z^n
=
\prod_{j=1}^dH_{\xi_j}(z),
\]
so the factor \((1-z)^{-1}\) is absent. Consequently, the saddle-point
equation becomes
\[
d\mu(r)=n,
\]
and the corresponding quadratic coefficient is
\[
V=dr\mu'(r).
\]
Nevertheless, the same estimates
\[
V\asymp\frac{dr}{\rho^2},
\qquad
\kappa=\rho^2V\asymp\min\{n,d\}
\]
remain valid. The two exceptional frequency regions are treated as in
the ball case, using the same parity identity. Moreover, the absence
of the factor \((1-z)^{-1}\) removes the corresponding rational factor
from \(R_{\varepsilon,\eta}\). The remainder of the saddle-point and
finite-difference argument follows the same scheme, with minor
simplifications. This gives
Proposition~\ref{prop:sphere-Gaussian-approximation}. More details on
the proof in the spherical case are sketched in Section~\ref{sec: spheres}.
\subsection{Notation}
\label{ssec:not}
\begin{enumerate}
\item For vectors $x,y\in\R^d$, we write
\[
  |x|=\Big(\sum_{j=1}^d x_j^2\Big)^{1/2},
  \qquad
  x\cdot y=\sum_{j=1}^d x_jy_j.
\]
\item For a finite set $A$, its cardinality is denoted by $|A|$.

\item We identify $\T^d=\R^d/\Z^d$ with $[-1/2,1/2)^d$ and use
the normalized Haar measure on $\T^d$. 
\item For $x\in\R$, we write
\[
  \|x\|_{\T}=\min_{k\in\Z}|x-k|.
\]
\item We set
\[
  e(s)=e^{2\pi i s},\qquad s\in\C.
\]

\item For $f\in\ell^1(\Z^d)$, its Fourier transform is
\[
  \widehat f(\xi)
  =\sum_{n\in\Z^d}f(n)e(n\cdot\xi),
  \qquad \xi\in\T^d.
\]
With this normalization, the Fourier transform extends to a unitary
map from $\ell^2(\Z^d)$ onto $L^2(\T^d)$.

\item Throughout the paper, $d$ denotes the dimension. All implicit
constants are independent of $d$ unless a dependence is explicitly
indicated. For nonnegative quantities $X,Y$, the notation
$X\lesssim Y$ means $X\le CY$ for an absolute constant $C$, while
$X\lesssim_{\alpha_1,\ldots,\alpha_k}Y$ allows $C$ to depend only on
the displayed parameters. We define $\asymp$ analogously and use the
same convention for $O(\cdot)$ notation. The values of absolute
constants $c,C>0$ may change from line to line.

\item We write
\[
  \mathbf h=\left(\frac12,\ldots,\frac12\right)\in\T^d.
\]
\item For $p\in[1,\infty]$, we abbreviate
\[
  \ell^p=\ell^p(\Z^d),\qquad
  \|f\|_p=\|f\|_{\ell^p(\Z^d)},
\]
and write $\|T\|_{p\to p}$ for the corresponding operator norm.
\end{enumerate}
\section*{Acknowledgments}
Jakub Niksiński was supported by the NSF CAREER grant DMS-2236493. Błażej Wróbel was supported by the National Science Centre, Poland, grant Sonata Bis 2022/46/E/ST1/00036

\section{Discrete Euclidean balls}
\label{sec: balls}

\subsection{Proof of Theorem \ref{thm:main}, assuming Proposition \ref{prop:euclidean-Gaussian-approximation}}
\label{ssec: main from prop}
\begin{proof}  
    The case $p= \infty$ is trivial, hence, by interpolation, it suffices to consider $p \in (1, 2]$. Since $\mathcal M_0^d=\Id$ and $d^4>Cd^2$ for all sufficiently
large $d$, by \eqref{eq:max ball large scale}, it is enough to prove
\begin{equation*}
%\label{eq:euclidean-bounded-radius-range}
 \left\|\sup_{1\leq n\leq d^4}|\mathcal M_n^d f|\right\|_{p}
 \lesssim_{p} \|f\|_{p}.
\end{equation*}
Consider the operator $J$, defined by the following formula
\begin{equation*}
    Jf(x)= e(x \cdot \mathbf{h}) f(x)= (-1)^{x_1+\ldots+x_d} f(x),
\end{equation*}
Note that then we have 

\begin{equation*}
    \widehat{Jf}(\xi)=\widehat{f}(\xi-\mathbf h), \qquad \widehat{J \mathcal G_r^dJf}(\xi)=\widehat{g}_r^d(\xi-\mathbf h)\widehat{f}(\xi),
\end{equation*}
moreover $J$ is an isometry on $\ell^p(\mathbb Z^d)$.
Fix $N$ sufficiently large in terms of $p$ and for every $1 \le n \le d^4$ consider the multiplier 
\begin{equation*}
%\label{eq:euclidean-Gaussian-approximation}
 \tilde{a}_{n,N}^d(\xi)
 :=\sum_{\nu=1}^{L_N}\alpha_{\nu,d,n}\widehat g_{r_{\nu,d,n}}^d(\xi)
 +\sum_{\nu=1}^{L_N}\beta_{\nu,d,n}
 \widehat g_{\widetilde r_{\nu,d,n}}^d(\xi-\mathbf h),
\end{equation*}
where the radii and coefficients are those from Proposition \ref{prop:euclidean-Gaussian-approximation}. Let $\mathcal{T}_{n,N}^d$ be a convolution operator corresponding to the multiplier $\tilde{a}_{n,N}^d$. Then Proposition \ref{prop:euclidean-Gaussian-approximation} implies the following pointwise bounds 
\[
\sup_{1\leq n\leq d^4}|\mathcal T_{n,N}^df|
 \lesssim_{N} 
  \sup_{0<r<1}|\mathcal G_r^df|
  +\sup_{0<r<1}|J\mathcal G_r^dJf|.
\]
The above inequality combined with
Theorem \ref{thm:gaussian} implies
\begin{equation} \label{lp bound of tilde A}
    \left\|\sup_{1\leq n\leq d^4}|\mathcal T_{n,N}^d f| \right\|_{p } \lesssim_{N,p} \|f\|_p.
\end{equation}
Moreover inequality \eqref{eq:euclidean-multiplier-remainder} combined with Parseval's theorem gives us 
\begin{equation}
\label{eq: l2 decay Parseval}
    \left\|\mathcal{M}_{n}^d-\mathcal T_{n,N}^d  \right\|_{2\to 2 } \lesssim_{N}  \min\{n,d\}^{-N}.
\end{equation}
Note also the trivial bound
\[
 \left\|\mathcal{M}_{n}^d-\mathcal T_{n,N}^d  \right\|_{1\to 1 } \lesssim_{N} 1,
\]
which follows from the $\ell^1$ contractivity of $\mathcal{M}_{n}^d$ and $ \mathcal G_r^d$ together with \eqref{eq:euclidean-Gaussian-coefficients}. Thus, if we choose $N=N(p)$ large enough, then \eqref{eq: l2 decay Parseval} interpolated with these trivial bound gives
\begin{equation} \label{lp bound for difference}
\left\|\mathcal{M}_{n}^d-\mathcal T_{n,N}^d \right\|_{p\to p } \lesssim_{N}  \min\{n,d\}^{-4}.
\end{equation}
Combining inequalities \eqref{lp bound of tilde A} and \eqref{lp bound for difference} we obtain
\begin{align*}
    \left\|\sup_{1\leq n\leq d^4}|\mathcal M_n^d f|\right\|_{p}
 &\le \left\|\sup_{1\leq n\leq d^4}|\mathcal T_{n,N}^d f| \right\|_{p} + \left\|\sup_{1\leq n\leq d^4}|(\mathcal M_{n}^d-\mathcal T_{n,N}^d)f| \right\|_{p} \\
 &\lesssim_{N,p}\|f\|_p+ \sum_{1 \le n \le d^4}\|(\mathcal M_{n}^d-\mathcal T_{n,N}^d)f\|_p \\
 &\lesssim_{N,p}\|f\|_p+ \sum_{1 \le n \le d^4} \min\{n,d \}^{-4}\|f\|_p \lesssim_{N,p} \|f\|_p.
\end{align*}
Since $N$ depends only on $p$, the above concludes the proof of Theorem \ref{thm:main}.

\end{proof}
\subsection{Definition of $\Theta$ and $H$ and generating function formula}
\begin{definition} \label{defn: defn of H}
    For $\theta \in \R$ and $|z|<1$ we define 
    \[
    H_\theta(z)= \sum_{k \in \Z} z^{k^2} e(k \theta),
    \]
    moreover for $t \in \R, r \in (0,1)$ we also define 
    \[
    \psi_{r}(t, \theta)= \frac{H_{\theta}(re(t))}{H_0(r)}.
    \]
\end{definition}
Note that $\Theta(z)=H_0(z)$. In the notation of Whittaker and Watson, we have
\(H_\theta(z)=\vartheta_3(\pi\theta,z)\), see
\cite[\S~21.11, pp.~463--464]{WhittakerWatson}. The usefulness of the functions $H_\theta(z)$, $\psi_r(t,\theta)$ comes from the observation that for every $\xi \in \T^d$ and $|z|<1$ we have 
\begin{align*}
    \sum_{n=0}^\infty |\mathcal{B}_n^{d}| m_n^d(\xi) z^n&= \sum_{n=0}^\infty \sum_{\substack{x \in \Z^d,\\|x|^2 \le n}}  z^n e(\xi \cdot x)\\ 
    &=  \sum_{\substack{x \in \Z^d}} e(\xi \cdot x)\sum_{n=|x|^2}^\infty z^n =\frac{\prod_{j=1}^d H_{\xi_j}(z)}{1-z}.
\end{align*}
Therefore by orthogonality (or Cauchy's formula and change of variables) for any $r \in (0,1)$ we have 
\begin{equation} \label{eq: Formula for multiplier}
   r^n |\mathcal{B}_n^{d}| m_n^d(\xi)= \int_\T \frac{\prod_{j=1}^d H_{\xi_j}(re(t))}{1-re(t)} e(-nt) \dd t,
\end{equation}
in particular 
\begin{equation*}
   r^n |\mathcal{B}_n^{d}| = \int_\T \frac{\Theta(re(t))^d}{1-re(t)} e(-nt) \dd t.
\end{equation*}
 We will apply the saddle-point method to the above formulas. This was briefly indicated at the end of \cite{MazoOdlyzko} for the case $n\asymp d$ and was used in \cite{NiksinskiWrobel} to compute $|\mathcal B_n^d|$ in the case $n \le cd$. 
\subsection{Saddle-point parameters and preliminary estimates}
\label{ssec: ball proof prelim}
For $r \in (0,1)$ put 
\[\rho=1-r\]
and define 
\[
\mu(r)= \frac{r \Theta'(r)}{\Theta(r)}.
\]
The saddle-point parameter $r=r_{d,n} \in (0,1)$ will be chosen to satisfy
\begin{equation} \label{eq: definition of r}
    n= \frac{r}{\rho}+d \mu(r).
\end{equation}
This choice is precisely the saddle-point condition. Namely, under the
parametrization \(z=re(t)\), it makes the linear term in the Taylor
expansion at \(t=0\) of the logarithm of the integrand in the preceding
formula for \(r^n|\mathcal B_n^d|\) vanish. Lemma below justifies the existence and uniqueness of $r_{d,n}$ and establishes preliminary estimates for quantities involving $\mu(r)$.

\begin{lemma} \label{lem: basic aproximation for mu}
    There is a unique $r_{d,n} \in (0,1)$ satisfying \eqref{eq: definition of r}. Moreover, for every $r \in (0,1)$ we have
    \[
    \Theta(r)\asymp \rho^{-1/2},\qquad
    \mu(r) \asymp \frac{r}{\rho}, \qquad r \mu'(r) \asymp \frac{r}{\rho^2}.
    \]
\end{lemma}
\begin{proof}
For the first part note that we have 
\begin{align*}
    r \mu'(r)&= \frac{\sum_{k \in \Z} k^4 r^{k^2}}{\Theta(r)}- \frac{\left(\sum_{k \in \Z} k^2 r^{k^2}\right)^2}{\Theta(r)^2} \\
    &= \frac{1}{2 \Theta(r)^2}\sum_{k,l \in \Z} r^{k^2+l^2}(k^2-l^2)^2>0
\end{align*}
and the above implies that 
\[
r \frac{d}{dr} \Big(\frac{r}{\rho}+d \mu(r) \Big)= \frac{r}{\rho^2}+d r \mu'(r)>0.
\]
Since $\lim_{r \to 0^+} \frac{r}{\rho}+ d \mu(r)=0$ and $\lim_{r \to 1^-} \frac{r}{\rho}+ d \mu(r)=\infty$, we conclude that there exists unique $r=r_{d,n}$ satisfying \eqref{eq: definition of r}.

For the second part, if $r \in (0,1/2)$, then the conclusion follows from the power series expansion of $\Theta(r)$. Assume now $r \in [1/2,1)$ and denote $r=e^{- \tau}$ so that
\[
 \tau= -\log r \asymp \rho.
\]
Applying Poisson summation formula we obtain

\[
 \Theta(e^{-\tau})
 =\sqrt{\frac{\pi}{\tau}}
 \left(1+2\sum_{n=1}^{\infty}e^{-\pi^2n^2/\tau}\right)
 \asymp \tau^{-1/2}\asymp\rho^{-1/2}
\]
and moreover,
\begin{align*}
\mu(e^{-\tau})&= \frac{\frac{\sqrt{\pi}}{2\tau^{3/2}}
+
\frac{\sqrt{\pi}}{\tau^{3/2}}
\sum_{n=1}^{\infty}
\left(
1-\frac{2\pi^2 n^2}{\tau}
\right)
e^{-\pi^2 n^2/\tau}}{\sqrt{\frac{\pi}{\tau}}
\left(
1+2\sum_{n=1}^{\infty} e^{-\pi^2 n^2/\tau}
\right)}\\
&= \frac{1}{2\tau} \cdot \frac{1+2\sum_{n=1}^{\infty}
\left(
1-\frac{2\pi^2 n^2}{\tau}
\right)
e^{-\pi^2 n^2/\tau}}{1+2\sum_{n=1}^{\infty} e^{-\pi^2 n^2/\tau}}.
\end{align*}
The above implies that for $\tau \le \delta$, where $\delta>0$ is a sufficiently small absolute constant, we have 
\[
\mu(r)= \mu(e^{-\tau}) \asymp \frac{1}{\tau} \asymp \frac{1}{\rho} \asymp \frac{r}{\rho}
\]
and
\[
r\mu'(r)= -re^{\tau}\frac{d}{d \tau}[\mu(e^{-\tau})] \asymp \frac{r}{\tau^2}  \asymp \frac{r}{\rho^2}.
\]
On the other hand if $\tau \ge \delta$, that is $ r \in [1/2, e^{-\delta}]$, the conclusion trivially holds, again by invoking the power series expansion of $\Theta(r)$.
\end{proof}
For the remainder of this subsection, let $r=r_{d,n}$ denote the unique solution of \eqref{eq: definition of r}. Let us introduce two parameters, which will be useful throughout the rest of the proof of Proposition \ref{prop:euclidean-Gaussian-approximation}. Define
 \begin{equation}
 \label{eq: def V kappa}
     V:=r \frac{d}{dr} \left(\frac{r}{\rho} + d \mu(r)\right)=\frac{r}{\rho^2}+ d r \mu'(r), \qquad  \kappa:=\rho^2V.
 \end{equation}
Lemma \ref{lem: basic aproximation for mu} implies that
 \[
n \asymp \frac{dr}{\rho}, \qquad V \asymp \frac{dr}{\rho^2},
 \]
 which after recalling $\rho=1-r$ yields
 \begin{equation} \label{eq: kappa, r , rho properties}
     r \asymp \frac{n}{n+d}, \qquad \rho \asymp \frac{d}{n+d}, \qquad \kappa=\rho^2 V \asymp \min\{n,d \}.
 \end{equation}
 In particular $\kappa \gtrsim 1$ and since $n \le d^4$ we get 
 \begin{equation} \label{eq: bound on V in terms of power of kappa}
     \rho^{-1} \asymp 1+\frac{n}{d} \lesssim \kappa^3, \qquad V= \frac{\kappa}{\rho^2} \lesssim \kappa^{7}.
 \end{equation}

Define
\begin{equation}\label{eq:euclidean-Phi}
 \Phi_{d,n}(t)
 =e(-nt)\frac{\rho}{1-re(t)}
 \left(\frac{H_0(re(t))}{H_0(r)}\right)^d,
\end{equation}
then by orthogonality we have
\[
\int_{-1/2}^{1/2} \Phi_{d,n}(t) \dd t=\rho \Theta(r)^{-d} r^n |\mathcal{B}_n^d|.
\]
As will be seen from the saddle-point expansion below, $V^{-1/2}$ is the natural scale of the main contribution to the coefficient integral, while $\rho$ is the scale on which the local expansion is valid. Thus $\kappa=\rho^2V$ measures the separation between these two scales and will serve as the main large parameter in the argument.

\subsection{Bounds for theta-function quotients at real arguments}
For every integer $L\geq2$ and all $t,\theta\in\mathbb \R$, put
\[
 Q_L(t,\theta)=\frac1{2L+1}\sum_{k=-L}^{L}e\big(tk^2+\theta k\big).
\]
We always have the trivial estimate $|Q_L(t,\theta)| \le 1$.
Moreover note that $|Q_L(t,\theta)|=1$ only for $(t,\theta) \in \Z^2 \cup \Big( (1/2,1/2)+ \Z^2 \Big)$. The point of the lemma below is to quantify the distance between $|Q_{L}(t,\theta)|$ and $1$ based on the distance of $(t,\theta)$ from $(0,0)$ and $(1/2,1/2)$ on $\T^2$. 
\begin{lemma}\label{lem:euclidean-quadratic-phase}
There is an absolute constant $c>0$ such that, for every integer $L\geq2$ and every
$t,\theta\in\R$ we have
\begin{equation}\label{eq:euclidean-finite-rigidity}
 1-|Q_L(t,\theta)|^2
 \geq c\min\left\{1,
 \min_{\varepsilon\in\{0,1\}}
 \left(
 L^4\|t-\varepsilon/2\|_{\T}^2
 +L^2\|\theta-\varepsilon/2\|_{\T}^2
 \right)\right\}.
\end{equation}
\end{lemma}

\begin{proof}
First we will prove that for every integer $M\geq2$ and every $a\in\mathbb \R$, one has
\begin{equation}\label{eq:euclidean-linear-phase}
 1-\left|\frac1M\sum_{j=0}^{M-1}e(ja)\right|^2
 \geq c\min\{1,M^2\|a\|_{\T}^2\}.
\end{equation}
We may assume that $a \in (-1/2, 1/2]$. Let $\varepsilon_0>0$ be a fixed small constant, to be determined later and let 
\begin{equation*}
    D_M(a)= \frac{1}{M} \sum_{j=0}^{M-1} e(ja)
\end{equation*}
denote the normalized Dirichlet kernel. 

In the case $M|a| \le \varepsilon_0$ we use the formula for Dirichlet's kernel and Taylor's theorem to get 
\begin{align*}
1-\left|D_M(a)\right|^2
&=1-\frac{\sin^2(\pi Ma)}{M^2\sin^2(\pi a)}=1-
\frac{\Big(\pi Ma-\frac{\pi^3M^3a^3}{6}
+O(M^5|a|^5)\Big)^2}
{M^2\Big(\pi a-\frac{\pi^3a^3}{6}
+O(|a|^5)\Big)^2}\\
&=1-
\frac{\Big(1-\frac{\pi^2M^2a^2}{6}
+O(M^4|a|^4)\Big)^2}
{\Big(1-\frac{\pi^2a^2}{6}
+O(|a|^4)\Big)^2}=1-
\frac{1-\frac{\pi^2M^2a^2}{3}
+O(M^4|a|^4)}
{1-\frac{\pi^2a^2}{3}
+O(|a|^4)}\\
&=\frac{\pi^2(M^2-1)a^2}{3}
+O(M^4a^4)=\frac{\pi^2M^2|a|^2}{3}
\left(1-\frac{1}{M^2}+O(\varepsilon_0^2)\right).
\end{align*}
Now, as long as $\varepsilon_0$ is small enough, the last expression is bounded from below by $cM^2|a|^2$ for some absolute constant $c>0$.

If $M|a| \in (\varepsilon_0, 1)$, then we use the fact that the function 
\[
 [0,1] \ni x \mapsto \frac{\sin(\pi x)}{x}
\]
is decreasing. This way, since $M \ge 2$, one obtains
\begin{align*}
    \left|D_M(a)\right|&= \left| \frac{\sin(\pi Ma)}{M \sin(\pi a)}  \right|= \frac{\sin(\pi M|a|)/(M|a|)}{ \sin(\pi|a|)/(|a|)} 
    \\
    &\le \frac{\sin(\pi M|a|)/(M|a|)}{ \sin(\pi M |a|/2)/(M|a|/2)}  = \cos(\pi M|a|/2) \le \cos(\pi\varepsilon_0/2) \le 1-c_0
\end{align*}
for some constant $c_0>0$. Lastly, if $M|a| \ge 1$, then using the inequality $\sin (\pi x) \ge 2x$, which holds for $x \in [0, 1/2]$, we get
\begin{align*}
    \left|D_M(a)\right|&= \left| \frac{\sin(\pi Ma)}{M \sin(\pi a)}  \right|= \frac{|\sin(\pi Ma)|}{ M\sin(\pi|a|)} \le  \frac{1}{2M |a|} \le \frac{1}{2}.
\end{align*}
In conclusion, after combining all the cases \eqref{eq:euclidean-linear-phase} is proved.
\par
Now we turn to the proof of \eqref{eq:euclidean-finite-rigidity}.
Let $\delta_0>0$ be a small constant to be determined later and let
\begin{equation*}
    \delta= 1- \left|Q_{L}(t, \theta) \right|^2, \qquad N=2L+1.
\end{equation*}
If $\delta>\delta_0$, then \eqref{eq:euclidean-finite-rigidity} holds trivially, hence we consider only the case $\delta \le \delta_0$.
First we will show that if $\delta_0$ is small enough, then 
\begin{equation} \label{eq: bound on t}
    \min_{\varepsilon \in \{0,1\}}\|t-\varepsilon/2 \|_{\T} \le C \frac{\sqrt{\delta}}{L^2}
\end{equation}
for some absolute constant $C>0$.

Expanding the square and changing variable $h=k-l$ in the spirit of Weyl we get 
\begin{equation} 
\label{eq: Weyl diff}
\begin{split}
    |Q_{L}(t, \theta)|^2&= \frac{1}{(2L+1)^2} \sum_{k,l=-L}^L e\big(t(k^2-l^2)+\theta(k-l) \big) \\
    & =\frac{1}{(2L+1)^2} \sum_{h=-2L}^{2L} e\big( th^2+\theta h\big) \sum_{l=\max(-L,-L-h)}^{\min(L,L-h)}e(2t hl).
\end{split}
\end{equation}
Using \eqref{eq: Weyl diff} and \eqref{eq:euclidean-linear-phase} one obtains
\begin{align*}
    \delta&=1-|Q_{L,\theta}(t)|^2 \ge  \frac{1}{N^2} \sum_{h=-2L}^{2L} \Big(N-|h| \Big) \Big(1- \left|D_{N-|h|}(2 ht) \right|\Big) \\
    &\ge \frac{1}{2N^2} \sum_{h=-2L}^{2L} \Big(N-|h| \Big) \Big(1- \left|D_{N-|h|}(2 ht) \right|^2\Big) \\
    &\ge  \frac{c}{L} \sum_{h=\lceil L/2 \rceil}^L  \min\{1, L^2 \|2t h \|_{\T}^2 \}.
\end{align*}

From the above we conclude that for sufficiently large constant $C>0$ the proportion of $h \in \{ \lceil L/2 \rceil, \ldots, L \}$ satisfying
\begin{equation*} 
    \| 2th \|_{\T} \le \frac{C\sqrt{\delta}}{L}
\end{equation*}
is greater than $3/4$, in particular there are two consecutive elements $h_0,h_0+1$ of that interval satisfying the inequality above. Then by triangle inequality we get 
\begin{equation*}
    \|2 t\|_{\T} \le  \|2 t(h_0+1)\|_{\T} +  \|2 th_0\|_{\T} \le \frac{2C\sqrt{\delta}}{L}.
\end{equation*}
Thus for some $\varepsilon \in \{0,1\}, m \in \mathbb Z, \eta \in \mathbb R$ we have 
\begin{equation*}
    t=  m+\varepsilon/2+ \eta, \quad \left |\eta \right| \le  \frac{C\sqrt{\delta}}{L}.
\end{equation*}
For $h_0$ as before we have 
\begin{equation*}
    \|2 \eta h_0 \|_{\T}= \|2th_0 \|_{\T} \le \frac{C\sqrt{\delta}}{L},
\end{equation*}
moreover $2 |\eta|h_0 \le 2C\sqrt{\delta_0}$, so if $\delta_0$ is small enough we get $|2 \eta h_0|<1/2$ and hence from previous inequality we see that
\begin{equation*}
    2|\eta|h_0 \le \frac{C\sqrt{\delta}}{L}.
\end{equation*}
Since $h_0 \ge L/2$, the inequality above implies
\begin{equation} \label{eq: t-eps bound}
    \|t-\varepsilon/2 \|_{\T}=|\eta| \le \frac{C'\sqrt{\delta}}{L^2},
\end{equation}
which establishes \eqref{eq: bound on t}. 

Now we will show that for the same choice of $\varepsilon$ as for $t$, we have 
\begin{equation*} 
%\label{theta bound}
    \| \theta -\varepsilon/2 \|_{\T} \le \frac{C'' \sqrt{\delta}}{L}.
\end{equation*}
Using the fact that $|Q_{L}(t,\theta)|=|Q_{L}(t+ 1/2,\theta+1/2)|$ one can see that 
$|Q_L(t, \theta)|=|Q_L(\eta, \beta)|$ for some $\beta \in [-1/2, 1/2]$ such that 
\begin{equation*}
    |\beta|= \|\theta -\varepsilon/2 \|_{\T}.
\end{equation*}
Note that we have 
\begin{equation}
\label{eq: eta, beta}
\begin{split}
    \delta=1-|Q_L(\eta,\beta)|^2&=\frac{1}{N^2} \sum_{k,l=-L}^{L}\Big(1-\cos\big(2\pi\eta(k^2-l^2)+2\pi\beta(k-l)\big) \Big) \\
    &\gtrsim \frac{1}{N^2} \sum_{k,l=-L}^{L} \left\| \eta(k^2-l^2)+\beta(k-l)\right \|_{\T}^2, 
\end{split}
\end{equation}
where at the end we have used the inequality $1- \cos(2\pi x) \gtrsim  \|x \|_{\T}^2$. Applying \eqref{eq: t-eps bound} and \eqref{eq: eta, beta} we obtain 
\begin{align*}
    \frac{1}{N^2} \sum_{k,l=-L}^{L} \left\| \beta(k-l)\right \|_{\T}^2 &\lesssim \frac{1}{N^2} \sum_{k,l=-L}^{L} \left\| \eta(k^2-l^2)+\beta(k-l)\right \|_{\T}^2 \\
    &+  \frac{1}{N^2} \sum_{k,l=-L}^{L}\left\| \eta(k^2-l^2) \right \|_{\T}^2 \lesssim \delta + L^4 |\eta|^2 \lesssim \delta.
\end{align*}
However, note that due to \eqref{eq:euclidean-linear-phase} and the comparison $1- \cos(2 \pi x) \lesssim \|x \|_{\T}^2$, which holds for all $x \in \mathbb R$, we get
\begin{align*}
    c\min\{1, L^2\|\beta \|_{\T}^2 \} \le1-|D_N(\beta)|^2&=  \frac{1}{N^2} \sum_{k,l=-L}^{L} \big(1-\cos\left(2\pi\beta(k-l) \right) \big)   \\ &\lesssim\frac{1}{N^2} \sum_{k,l=-L}^{L} \left\| \beta(k-l)\right \|_{\T}^2 \lesssim \delta,
\end{align*}
the above combined with definition of $\beta$ and \eqref{eq: t-eps bound} gives 
\begin{equation*}
    L^4\|t - \varepsilon/2 \|_{\T}^2+ L^2 \| \theta - \varepsilon/2 \|_{\T}^2 \lesssim \delta
\end{equation*}
for some $\varepsilon \in \{0,1\}$, which concludes the proof.
\end{proof}
Now we will use Lemma \ref{lem:euclidean-quadratic-phase} to deduce a bound of a similar shape for $\psi_r(t,\theta)$, which will be crucial to handle error terms in the proof of Lemma \ref{lemma:euclidean-coefficient-estimate}, where we shall estimate $|\mathcal{B}_n^d|$.
\begin{corollary} \label{corollary: psi decay}
 For $r \in (0,1)$, $t,\theta \in \R$ let 
 \[
 \Delta_r(t,\theta)= \min_{\varepsilon \in \{0,1 \}} \left\{ \frac{\|t-\varepsilon/2\|_{\T}^2}{\rho^2}+ \frac{\| \theta-\varepsilon/2 \|_{\T}^2}{\rho} \right\}.
 \]
 Then there exists an absolute constant $c>0$ such that 
 \[
 |\psi_r(t,\theta)| \le \exp\big(-cr \min\{1,\Delta_{r}(t,\theta)\}\big)
 \]
\end{corollary}
Even though the exponent on the right-hand side of the above inequality might be small, in application we will consider a product of $d$ such terms for various different $\theta \in \R$ which will give us sufficient decay for applications.
\begin{proof}
    Due to the inequality $1-x \le e^{-x}$ it suffices to show that 
    \[
    1-|\psi_r(t,\theta)|^2 \gtrsim r \min\{1,\Delta_{r}(t,\theta)\}.
    \]
    Note that we have the following identity
    \begin{equation*}
        1-|\psi_r(t,\theta)|^2=\frac{1}{2\Theta(r)^2} \sum_{k,l \in \Z} r^{k^2+l^2} \big|e\big(\ t(k^2-l^2)+\theta(k-l)\big)-1 \big|^2.
    \end{equation*}
    
    If $r \in(0,1/2)$ then using $\Theta(r) \asymp 1$ and discarding all terms except for $(k,l) =\pm(1,0) $ gives 
    \begin{align*}
        1-|\psi_r(t,\theta)|^2 &\gtrsim r \Big(|e(t+\theta)-1|^2 +|e(t-\theta)-1|^2 \Big) \\
        &=4r\big(1-\cos(2\pi t) \cos(2\pi \theta) \big)=:4rS(t,\theta).
    \end{align*}
    Note that $S(t,\theta)$ as a function on $\T^2$ has zeros only at $(0,0)$ and $ (1/2,1/2)$. Therefore standard compactness and Taylor approximation arguments imply
    \[
    S(t,\theta) \gtrsim  \min_{\varepsilon \in \{0,1 \}} \left(\|t-\varepsilon/2\|_{\T}^2+ \| \theta-\varepsilon/2 \|_{\T}^2 \right)
    \]
    and the conclusion for $r \in (0,1/2)$ follows. 

    Now consider $r \in [1/2,1)$ (so that $-\log r \asymp 1-r=\rho$) and set $L=\lceil \rho^{-1/2} \rceil.$ Then by $\Theta(r) \asymp \rho^{-1/2}$ (see Lemma \ref{lem: basic aproximation for mu}) for every $|k| \le L$ we obtain
    \[
    \frac{r^{k^2}}{\Theta(r)}=\frac{e^{k^2\log r }}{\Theta(r)} \gtrsim L^{-1}.
    \]
     Finally, using Lemma \ref{lem:euclidean-quadratic-phase} we see that 
    \begin{align*}
        1-|\psi_r(t,\theta)|^2&=\frac{1}{2\Theta(r)^2} \sum_{k,l \in \Z} r^{k^2+l^2} \big|e\big(\ tk^2+\theta k\big)-e\big(\ tl^2+\theta l\big) \big|^2 \\
        &\gtrsim \frac{1}{(2L+1)^2} \sum_{k,l=-L}^L \big|e\big(\ tk^2+\theta k\big)-e\big(\ tl^2+\theta l\big) \big|^2  \\
        &=2\bigl(1-|Q_L(t,\theta)|^2\bigr) \\
        &\gtrsim \min\left\{1,
 \min_{\varepsilon\in\{0,1\}}
 \left(
 L^4\|t-\varepsilon/2\|_{\T}^2
 +L^2\|\theta-\varepsilon/2\|_{\T}^2
 \right)\right\} \\
 &\asymp r \min\{1,\Delta_{r}(t,\theta)\},
    \end{align*}
    which concludes the proof.
\end{proof}
\subsection{Saddle-point analysis of the normalizing factor}
We now apply Corollary \ref{corollary: psi decay} to approximate $|\mathcal B_{n}^d|$ by the saddle-point method. Lemma \ref{lemma:euclidean-coefficient-estimate} generalizes \cite[Theorem~3.4, (3.4)]{NiksinskiWrobel}. Throughout this section, $r \in (0,1)$ will denote the unique number satisfying \eqref{eq: definition of r}. Recall 
definition of $\Phi_{d,n}(t)$ from \eqref{eq:euclidean-Phi} and let
\begin{equation} \label{eq: defn of I_{d,n} }
    I_{d,n}:= \int_{-1/2}^{1/2} \Phi_{d,n}(t) \dd t =\rho \Theta(r)^{-d} r^n |\mathcal{B}_n^d|.
\end{equation}
the second equality follows from orthogonality. Before saddle point computation we will obtain a bound on $|\Phi_{d,n}(t)|$, which will be also useful when completing the proof of Proposition \ref{prop:euclidean-Gaussian-approximation} in the next subsections. Here we recall the definitions of $V$ and $\kappa$ given in \eqref{eq: def V kappa}. 
\begin{lemma} \label{lemma:decay of Phi} For every constant $c_1 \in (0,1/2)$ there exists $c_2>0$ such that we have
\begin{equation}\label{eq:euclidean-Phi-local-decay}
 |\Phi_{d,n}(t)|\leq e^{-c_2Vt^2},
 \qquad \text{for} \ |t|\leq c_1\rho.
\end{equation}
and
\begin{equation}\label{eq:euclidean-Phi-global-decay}
 |\Phi_{d,n}(t)|\leq e^{-c_2\kappa}, \qquad \text{for} \ c_1 \rho \le |t| \le 1/2.
\end{equation}
\end{lemma}
\begin{proof}
Suppose first that $|t|\leq c_1\rho$. {Since $c_1<1/2$ from the definition of $ \Delta_r(t,\theta)$, we have
\[
 \Delta_r(t,0)\geq
 \min\left\{\frac{t^2}{\rho^2},\frac14\right\}=\frac{t^2}{\rho^2}
\]
and Corollary \ref{corollary: psi decay} gives
\[
 \left|\frac{H_0(re(t))}{H_0(r)}\right|^d
 \leq
 \exp\left(-c\frac{drt^2}{\rho^2}\right)
 \leq e^{-c_2Vt^2}.
\]}
Moreover,
\[
 \left|\frac{\rho}{1-re(t)}\right|^2
 =\frac{\rho^2}{\rho^2+2r(1-\cos(2\pi t))}\leq1
\]
and this proves \eqref{eq:euclidean-Phi-local-decay}.

{If $c_1\rho\leq |t|\leq1/2$, then
$\min\{1,\Delta_r(t,0)\}\ge \min\{c_1^2,\frac14\}$. Hence Corollary
\ref{corollary: psi decay}, together with the preceding bound for the
rational factor, gives
\[
 |\Phi_{d,n}(t)|\leq e^{-c \min\{c_1^2,\frac14\} dr}\leq e^{-c_2\kappa},
\]
because $\kappa\asymp dr$.} This proves \eqref{eq:euclidean-Phi-global-decay}.

\end{proof}
We are now ready to give our approximation for $I_{d,n}.$
\begin{lemma}
\label{lemma:euclidean-coefficient-estimate}
  For every $d\geq1$ and every integer $1\leq n\leq d^4$, we have
    \[
    I_{d,n} \asymp V^{-1/2}.
    \]
\end{lemma}
\begin{proof}
  Let \(c_\Phi>0\) be a
sufficiently small absolute constant. First, we will give an upper bound for the 3rd-order derivatives of expressions that will appear soon. Note that for all $t \in [-c_\Phi\rho, c_\Phi\rho]$ we have
\begin{equation}\label{eq:euclidean-geometric-third-derivative}
 \left|\frac{d^3}{d t^3}\log\frac\rho{1-re(t)}\right|
 \lesssim \frac r{\rho^3}.
\end{equation}

Now we will prove the same inequality for \(\log H_0(re(t))\), that is
\begin{equation}\label{eq:euclidean-H_0 third derivative bound}
 \left|\frac{d^3}{d t^3} \log H_0(re(t))\right|
 \lesssim \frac r{\rho^3}
\end{equation}
for all $t \in [-c_\Phi\rho, c_\Phi\rho]$. It is worth pointing out that $H_0$ does not vanish inside the complex unit disk due to Jacobi's product formula; hence the above logarithm exists.
Fix $r_0<1$ sufficiently close to
one. We will consider separately the cases $r\le r_0$ and $r_0<r<1$. In the first case, $\rho \asymp 1,$ and due to $H_0(z)$ being holomorphic and nonvanishing inside the complex unit disk, one can see that \eqref{eq:euclidean-H_0 third derivative bound} trivially holds.
Now suppose that $r_0<r<1$. Write
\[
 re(t)=e^{-\pi a},
 \qquad
 a=\frac{-\log r-2\pi i t}{\pi}.
\]
Then $\operatorname{Re}a\asymp\rho$, $|a|\asymp\rho$, and hence
\[
 \operatorname{Re}\frac1a
 =\frac{\operatorname{Re}a}{|a|^2}\asymp \rho^{-1}.
\]
After a standard application of the Poisson summation formula (or, more precisely, Jacobi's imaginary transformation; see, for instance,
\cite[\S~21.51, pp.~475--476]{WhittakerWatson}), we get
\[
 H_0(e^{-\pi a})
 =\frac{1}{\sqrt{a}}
 \sum_{\ell\in\mathbb Z}
 \exp\left(-\frac{\pi\ell^2}{a}\right),
\]
where the branch of the square root is chosen so that it agrees with the standard square root on positive reals. Inserting back the value of $a$, we have
\[
 H_0(re(t))
 =\sqrt{\frac{\pi}{-\log r-2\pi i t}}
 \sum_{\ell\in\mathbb Z}
 \exp\left(-\frac{\pi^2\ell^2}{-\log r-2\pi i t}\right).
\]
Note that because
\[
 |e^{-\pi \ell^2/a}|
 =e^{-\pi \ell^2 \operatorname{Re}(1/a)}
 \le e^{-c\ell^2/\rho},
\]
the terms of the last series expansion for $H_0(re(t))$, together with all their fixed-order derivatives, converge rapidly to zero in $\ell$ (faster than exponentially).

If $r_0$ is sufficiently close to 1, due to this rapid convergence, one has
\[
 \left|\frac{d^j}{d t ^j} H_0(re(t))\right|
 \lesssim_j \rho^{-j-1/2}
\]
for all integers $j \ge 0$
and
\[
 |H_0(re(t))| \gtrsim \rho^{-1/2},
\]
and therefore \eqref{eq:euclidean-H_0 third derivative bound} holds.

Define
\[
 K(t):=\log\frac\rho{1-re(t)}
 +d\log\frac{H_0(re(t))}{H_0(r)}.
\]
Combining inequalities
\eqref{eq:euclidean-geometric-third-derivative} and
\eqref{eq:euclidean-H_0 third derivative bound} gives, uniformly for
$t \in [-c_\Phi\rho, c_\Phi\rho]$, the following bound:
\begin{equation*}
%\label{eq:euclidean-third-log-derivative}
 |K'''(t)|
 \lesssim\frac{dr}{\rho^3}\asymp\frac V\rho.
\end{equation*}
For $t \in [-c_\Phi\rho, c_\Phi\rho]$, we define
the logarithm of \(\Phi_{d,n}(t)\) by
\[
 \log\Phi_{d,n}(t):=K(t)-2\pi i nt.
\]
Note that at $t=0$ we have
\[
 K'(0)=2\pi i\left(\frac r\rho+d\mu(r)\right)=2\pi in,
 \qquad
 K''(0)=-4\pi^2r\frac d{dr}
 \left(\frac r\rho+d\mu(r)\right)=-4\pi^2V.
\]
Thus Taylor's formula gives
\[
 K(t)=2\pi i nt-2\pi^2Vt^2
 +\frac{t^3}{2}\int_0^1(1-s)^2K'''(st)\dd s.
\]
This way we obtain
\begin{equation}\label{eq:euclidean-saddle-expansion}
 \log\Phi_{d,n}(t)
 =-2\pi^2Vt^2+O\left(\frac V\rho|t|^3\right),
 \qquad \text{for } |t|\leq c_\Phi\rho.
\end{equation}

Fix $B\geq1$ sufficiently large. If $|v|\leq B$ and $\kappa=\rho^2V$ is large enough in terms
of $B$, we have
\[
 \left|\frac v{\sqrt V}\right|
 =\frac{\rho|v|}{\sqrt\kappa}
 \leq c_\Phi\rho.
\]
We remark in passing that if $\kappa$ is large, then $V$ is large too.
Thus, for such $B$ and $\kappa$,
\eqref{eq:euclidean-saddle-expansion} gives
\begin{equation}\label{eq:euclidean-Gaussian-limit}
 \Phi_{d,n}\left(\frac v{\sqrt V}\right)
 =e^{-2\pi^2v^2}\left(1+O_B(\kappa^{-1/2})\right).
\end{equation}
Now, changing the variables $v=t\sqrt V$, we have
\begin{equation*}
%\label{eq:euclidean-scaled-coefficient-integral}
 \sqrt V I_{d,n}
 =\int_{-\frac12\sqrt V}^{\frac12\sqrt V}
 \Phi_{d,n}\left(\frac v{\sqrt V}\right)\dd v.
\end{equation*}
Therefore we obtain
\begin{align*}
 &\left|\sqrt V I_{d,n}-(2\pi)^{-1/2}\right|
 =\left|
 \int_{-\frac12\sqrt V}^{\frac12\sqrt V}
 \Phi_{d,n}\left(\frac v{\sqrt V}\right)\dd v
 -\int_{\R}e^{-2\pi^2v^2}\dd v
 \right| \\
 &\leq
 \left|
 \int_{-B}^{B}
 \Phi_{d,n}\left(\frac v{\sqrt V}\right)\dd v
 -\int_{-B}^{B}e^{-2\pi^2v^2}\dd v
 \right|
 +\int_{B\leq|v|\leq{c_\Phi}\rho\sqrt V}
 \left|\Phi_{d,n}\left(\frac v{\sqrt V}\right)\right|\dd v \\
 &\quad+
 \int_{{c_\Phi}\rho\sqrt V\leq|v|\leq\frac12\sqrt V}
 \left|\Phi_{d,n}\left(\frac v{\sqrt V}\right)\right|\dd v
 +\int_{|v|\geq B}e^{-2\pi^2v^2}{\dd v} \\
 &\lesssim_{B}
 \kappa^{-1/2}+e^{-cB^2}+\sqrt V e^{-c\kappa}+e^{-cB^2},
\end{align*}
where, in the second inequality above,
the first term was bounded using \eqref{eq:euclidean-Gaussian-limit},
the second one using \eqref{eq:euclidean-Phi-local-decay},
and the third one using \eqref{eq:euclidean-Phi-global-decay}.
Recall from \eqref{eq: bound on V in terms of power of kappa} that
$\sqrt V\lesssim\kappa^{7/2}$. Thus, if $B$ is large enough and
$\kappa>C_B$ for some large enough constant depending on $B$, we obtain
\begin{equation}\label{eq: I_d,n comparable with V^{-1/2}}
 I_{d,n}\asymp V^{-1/2}.
\end{equation}

It remains to treat the case $\kappa\leq C_B$. Since
$\kappa\asymp\min\{n,d\}$, we have $\min\{n,d\}\lesssim_B1$.
If $d\lesssim_B1$, then there are only finitely many possible pairs
$(n,d)$, and the conclusion follows after adjusting the implicit
constants. We may therefore assume that $d$ is sufficiently large in
terms of $B$. Then $n\lesssim_B1$, $n\leq d$, and
\eqref{eq: kappa, r , rho properties} gives
\[
 r\asymp_Bd^{-1},\qquad \rho\asymp_B1,\qquad V\asymp_B1.
\]
Every $x\in\mathcal B_n^d$ has at most $n$ nonzero coordinates and
$|x_j|\leq\sqrt n,$ for $j=1,\ldots,d$. On the other hand, all vectors with exactly $n$
coordinates equal to $1$ or $-1$ belong to $\mathcal B_n^d$. Hence
\[
 2^n\binom dn
 \leq|\mathcal B_n^d|
 \leq\sum_{j=0}^n\binom dj(2\lfloor\sqrt n\rfloor{+1})^j
 \lesssim_Bd^n,
\]
so {that we have} $|\mathcal B_n^d|\asymp_Bd^n$. Moreover,
$\Theta(r)=1+O(r)$ and $dr\asymp_B1$, whence
$\Theta(r)^d\asymp_B1$ and $r^n|\mathcal B_n^d|\asymp_B1;$ {therefore we obtain}
\[
 I_{d,n}
 =\rho\Theta(r)^{-d}r^n|\mathcal B_n^d|
 \asymp_B1\asymp_BV^{-1/2}.
\]
Since $B$ was fixed as an absolute constant, this proves
\eqref{eq: I_d,n comparable with V^{-1/2}} also in the bounded-$\kappa$
case. Thus the estimate holds for every $1\leq n\leq d^4$.
\end{proof}

    \subsection{Bounds for theta-function quotients at complex arguments}
    Our next goal will be to estimate $|H_\theta(z)/H_0(z)|$ for complex numbers $z\in\C$, which are not necessarily real. These estimates for $z \not\in (0,1)$ are not necessary for the proof of Lemma \ref{lemma:euclidean-coefficient-estimate}. However, they will play a crucial role in estimating error terms when applying the saddle point method to the ball multiplier $ m_n^d(\xi)$. 
    
    Recalling   that \(H_\theta(z)=\vartheta_3(\pi\theta,z)\) in the notation of Whittaker and Watson and using Jacobi's product formula (see for instance \cite[\S\S~21.3 and 21.42, pp.~469--470 and 472--473]{WhittakerWatson}) we have 
    \[
    H_0(z)=\prod_{l=1}^\infty(1-z^{2l})(1+z^{2l-1})^2,
    \]
    in particular $H_0$ doesn't vanish inside the unit disk. Moreover the same Jacobi product formula gives us
    \begin{equation} \label{eq:product for H quotient}
        \frac{H_\theta(z)}{H_0(z)}= \prod_{\ell=1}^\infty\big(1-\omega_l(z) \sin^2(\pi \theta) \big),
    \end{equation}
    where
    \begin{equation*}
        \omega_{\ell}(z)= \frac{4z^{2\ell-1}}{(1+z^{2\ell-1})^2},
    \end{equation*}
    in particular the product converges uniformly on compact subsets of the open unit disk.
     Notice that for $z=r \in (0,1)$ we have $|H_\theta(r)/H_0(r)| \le \exp (- c \frac{r}{1-r} \sin^2 (\pi \theta))$. The next lemma extends this estimate to complex \(z\) which lie sufficiently close to the positive real axis.
    \begin{lemma} \label{lemma: bound for quotient of H(z)}
        Let \(\arg z\in(-\pi,\pi]\) denote the principal argument.  There are
absolute constants $c_0,c>0$ such that, for every $\theta\in\R$ and every
$z\in\mathbb C$ satisfying $0<|z|<1$ and
\begin{equation*}
%\label{eq:euclidean-Stolz-region}
 |\arg z|\leq c_0(1-|z|),
\end{equation*}
we have
\begin{equation} \label{eq: sum omega ell bound}
    \sum_{\ell=1}^\infty |\omega_\ell(z)| \lesssim \frac{|z|}{1-|z|}.
\end{equation}
and
\begin{equation}\label{eq:euclidean-Stolz-decay}
 \left|\frac{H_\theta(z)}{H_0(z)}\right|
 \leq\exp\left(-c\frac{|z|}{1-|z|}\sin^2(\pi \theta)\right).
\end{equation}
\end{lemma}
\begin{proof}
Let us introduce a few parameters
\begin{equation*}
    q=|z|,
    \qquad
    \delta=1-q,
    \qquad
    s_\theta=\sin^2(\pi \theta), \qquad \alpha_\ell
    =
    \frac{4q^{2\ell-1}}{(1+q^{2\ell-1})^2}.
\end{equation*}
Note that we have
\begin{equation}
    \alpha_\ell \asymp q^{2\ell-1}, \qquad
    \sum_{\ell\geq1}\alpha_\ell \asymp
    \frac{q}{1-q}.
    \label{eq:alpha-total}
\end{equation}
Fix a large absolute constant $B \ge 1$, to be specified in the course of the proof. For $m=2\ell-1$, write 
\begin{equation*}
    y=q^m,
    \qquad
    \varphi=\frac{m\arg z}{2 \pi}.
\end{equation*}
We will split the factors of product depending on whether \(m\delta\leq B\) or \(m\delta>B\).

First, suppose that $m \delta \le B$. By our assumption on $z$ we have
\begin{equation*}
    |\varphi|=\frac{m|\arg(z)|}{2 \pi} \le c_0 m \delta/(2\pi) \le c_0B/(2\pi).
\end{equation*}
Moreover a short calculation shows that 
\begin{equation*}
    \frac{\omega_\ell(z)}{\alpha_\ell}-1=
    \frac{( e(\varphi)-1)(1-y^2e(\varphi))}{(1+ye(\varphi))^2}.
\end{equation*}
Taking $c_0$ small enough in terms of $B$, we may assume that
$\operatorname{Re}e(\varphi)\geq1/2$. Thus, 
\[
 |1+ye(\varphi)|\geq
 \operatorname{Re}(1+ye(\varphi))\geq1,
\]
and consequently
\begin{equation*}
    \Big|\frac{\omega_\ell(z)}{\alpha_\ell}-1 \Big| \lesssim |e(\varphi)-1|\lesssim |\varphi| \lesssim c_0B
\end{equation*}
uniformly in $y \in (0,1)$. In particular for small enough $c_0$ we obtain
\begin{equation} \label{eq: omega ell small m bound}
    \operatorname{Re}\omega_\ell(z)
    \geq
    \frac{9}{10}\alpha_\ell,
    \qquad
    |\omega_\ell(z)|
    \leq
    \frac{11}{10}\alpha_\ell
\end{equation}
and hence
\begin{align*}
    |1-s_\theta\omega_\ell(z)|^2
    &=
    1
    -2s_\theta\operatorname{Re}\omega_\ell(z)
    +s_\theta^2|\omega_\ell(z)|^2
    \\
    &\leq
    1
    -\frac95 s_\theta\alpha_\ell
    +\frac{121}{100}s_\theta\alpha_\ell
    \\
    &\leq
    1-\frac12s_\theta\alpha_\ell.
\end{align*}
Therefore
\begin{equation}
    \log|1-s_\theta\omega_\ell(z)|
    \leq
    -\frac14s_\theta\alpha_\ell
    \label{eq:main-factor}
\end{equation}
holds whenever $m\delta\leq B$.

Consider now $m\delta \ge B$. Note that for any $m=2\ell-1$ satisfying this condition we have
\[
q^m=e^{m \log q} \le e^{-m(1-q)}=e^{-m\delta} \le e^{-B}, 
\]
in particular 
\begin{equation} \label{eq: omega ell bound for m big}
    |\omega_{\ell}(z)| \lesssim q^{2\ell-1} \lesssim e^{-B}.
\end{equation}
Hence, if we choose $B$ large enough, then the left-hand side is smaller than $1/2$. Note that $B$ is chosen here
with respect to universal constants in \eqref{eq: omega ell bound for m big}. This affects the choice of $c_0$ in the case $m\delta \le B,$ however, all these constants remain universal.

At this point \eqref{eq: sum omega ell bound} easily follows from \eqref{eq:alpha-total}, \eqref{eq: omega ell small m bound}, and \eqref{eq: omega ell bound for m big}.

It remains to prove \eqref{eq:euclidean-Stolz-decay}. To establish this we will estimate the logarithm of the absolute value of the product from 
\eqref{eq:product for H quotient}. Namely, using \eqref{eq:main-factor} and \eqref{eq: omega ell bound for m big} (recall that there $|\omega_l(z)|<1/2$), followed by \eqref{eq:alpha-total} we have
\begin{align*}
 \sum_{\ell=1}^\infty \log|1-s_\theta\omega_\ell(z)|
 &= \sum_{{\ell}\le B(2\delta)^{-1}+{\frac12}}
 \log|1-s_\theta\omega_\ell(z)|
 + \sum_{{\ell}>B(2\delta)^{-1}+{\frac12}}
 \log|1-s_\theta\omega_\ell(z)| \\
 &\le -\frac14
 \sum_{{\ell}\le B(2\delta)^{-1}+{\frac12}}
 s_\theta\alpha_\ell
 +O\left(
 \sum_{{\ell}>B(2\delta)^{-1}+{\frac12}}
 q^{2{\ell}-1}s_\theta
 \right) \\
 &\le -\frac14
 \sum_{{\ell}=1}^\infty s_\theta\alpha_\ell
 +O\left(
 \sum_{{\ell}>B(2\delta)^{-1}+{\frac12}}
 q^{2{\ell}-1}s_\theta
 \right) \\
 &\le -{c'}\frac{q}{1-q}s_\theta
 +O\left(e^{-B/2}\frac{q}{1-q}s_\theta\right).
\end{align*}
In particular if $B$ is properly chosen, then we get 
\[
 \sum_{\ell=1}^\infty \log|1- s_\theta \omega_\ell(z)| \le -{c} \frac{q}{1-q}s_\theta.
\]

which, after taking exponentials, concludes the proof. 
\end{proof}
Using Lemma \ref{lemma: bound for quotient of H(z)} we now derive a bound for derivatives of products of quotients of theta series.
\begin{corollary} \label{corllary:euclidean-product-derivatives}
    Let $r \in (0,1), \ \xi \in \T^d$ and define 
    \begin{equation} \label{eq: definition of U_r(xi)}
        U_r(\xi)= \frac{r}{\rho} \sum_{j=1}^d \sin^2(\pi \xi_j).
    \end{equation}
    Take $t \in [-c_1\rho, c_1\rho]$, where $c_1$ is an absolute constant. Then, there exists a universal constant $c>0$ such that for every integer $m\ge 0$ we have 
    \begin{equation*} 
    %\label{eq:euclidean-product-derivatives}
        \rho^m\left|
 \partial_t^m
 \prod_{j=1}^d\frac{H_{\xi_j}(re(t))}{H_0(re(t))}
 \right|
 \lesssim_m(1+U_r(\xi))^m e^{-cU_r(\xi)}.
    \end{equation*}
    In particular the left-hand side is bounded by a constant depending only on $m$.
\end{corollary}
\begin{proof}
   For $\theta \in \T$ denote \(u_{r,\theta}=(r/\rho)\sin^2(\pi\theta)\), moreover for \(w\in\mathbb C\) such
that \(re^{-2 \pi\operatorname{Im}w}<1\) define
\[
 F_{r,\theta}(w)=\frac{H_\theta(re(w))}{H_0(re(w))}.
\]
Consider any $w \in \C$ such that $|w-t| \le c_1 \rho$.
If $c_1>0$ is sufficiently small, then, for
$z=r\,e(w) $ we obtain
\[
 |z| = r+O(c_1 \rho), \qquad
 |\arg z|\leq 2\pi|t|+2\pi|\operatorname{Re}(w-t)|
 \leq 4\pi c_1\rho\leq c_0(1-|z|),
\]
where $c_0$ is the constant from Lemma \ref{lemma: bound for quotient of H(z)}.
Thus, applying this lemma we get
\begin{equation*}
 |F_{r,\theta}(w)|\leq e^{-cu_{r,\theta}}.
\end{equation*}
Consequently, using Cauchy's estimate on the complex disk of radius \(c_1\rho\) around $t,$ we get
\begin{equation}
\label{eq:euclidean-complex-F-decay}
 \rho^m|\partial_t^mF_{r,\theta}(t)|
 \lesssim_m e^{-cu_{r,\theta}},\qquad m\ge 0.
\end{equation}

If $u_{r,\theta}\leq1$, then \eqref{eq:product for H quotient} and \eqref{eq: sum omega ell bound} give
\[
 |F_{r,\theta}(w)-1|\le \prod_{\ell=1}^\infty\big(1+\sin^2(\pi \theta) |\omega_l(z) |\big)-1
 \leq\exp\left(\sin^2(\pi \theta)\sum_\ell|\omega_\ell(z)|\right)-1
 \lesssim u_{r,\theta}.
\]
Hence, if $u_{r, \theta} \le 1$, using again Cauchy's estimate we get
\[
 \rho^m|\partial_t^mF_{r,\theta}(t)|
 \lesssim_m u_{r,\theta},
\]
but this time for every $m\ge 1.$

  After potentially increasing the implied constants, combining the above estimates, we obtain
\begin{equation}\label{eq:euclidean-one-coordinate-derivative}
 \rho^m|\partial_t^mF_{r,\theta}(t)|
 \lesssim_m u_{r,\theta}e^{-c u_{r,\theta}},\qquad m\ge 1.
\end{equation}

Recall that $U_r(\xi)$ was defined by \eqref{eq: definition of U_r(xi)} and set $u_j= \frac{r}{\rho}\sin^2(\pi\xi_j)$, so that
$U_r(\xi)=\sum_j u_j$.  Using the Leibniz formula we obtain
\[
 \partial_t^m\prod_{j=1}^dF_{r,\xi_j}(t)
 =\sum_{m_1+\cdots+m_d=m}
 \frac{m!}{m_1!\cdots m_d!}
 \prod_{j=1}^d\partial_t^{m_j}F_{r,\xi_j}(t).
\]
For each multi-index \((m_1,\ldots,m_d)\) in the sum above, denote
\(S=S(m)=\{j\in[d]:m_j\geq1\}\).  Equations \eqref{eq:euclidean-complex-F-decay} and
\eqref{eq:euclidean-one-coordinate-derivative} bound the corresponding term by
\[
 C_m\,\rho^{-m}e^{-cU_r(\xi)}\prod_{j\in S}u_j.
\]
 Note that, for fixed \(m\), the number of positive choices of the \(m_j\)'s corresponding to a fixed set \(S\) of size \(s\le m\) is \(\binom{m-1}{s-1}\). In particular, it depends only on \(s\) and \(m\). This way we finally get
\begin{align*}
    \rho^m\left|
 \partial_t^m
 \prod_{j=1}^d\frac{H_{\xi_j}(re(t))}{H_0(re(t))}
 \right|
 &\lesssim_m e^{-cU_r(\xi)} \sum_{s \le m} \sum_{\substack{S\subseteq [d], \\ |S|=s}} \prod_{j\in S}u_j \\
 & \le e^{-cU_r(\xi)} \sum_{s \le m} \frac{\big(\sum_{j\in [d]} u_j \big)^s}{s!} \\
 &\lesssim_m (1+U_r(\xi))^m e^{-c U_r(\xi)},
\end{align*}
which concludes the proof.
\end{proof}

\subsection{ Gaussian approximation of the ball multiplier - proof of Proposition \ref{prop:euclidean-Gaussian-approximation}}

We are finally ready to prove Proposition \ref{prop:euclidean-Gaussian-approximation}. For further reference recall that
\[
\mathbf{h}= (1/2,\ldots,1/2) \in \T^d.
\]
\begin{proof}
    Consider any $\eta \in \T^d$ and recall the formula \eqref{eq: Formula for multiplier} for the multiplier 
    \[
    r^n|\mathcal{B}_n^{d}| m_n^d(\eta)= \int_\T \frac{\prod_{j=1}^d H_{\eta_j}(re(t))}{1-re(t)} e(-nt) \dd t.
    \]
Note that the identity \(k^2\equiv k\pmod 2\) gives
\begin{equation}\label{eq:euclidean-parity}
 H_{\theta+1/2}(z)=H_\theta(-z).
\end{equation}
For \(\varepsilon\in\{-1,1\}\) we set
\begin{equation*}
%\label{eq:euclidean-R}
 R_{\varepsilon,\eta}(t)
 =
 \frac{1-re(t)}{1-\varepsilon re(t)}
 \prod_{j=1}^d
 \frac{H_{\eta_j}(re(t))}{H_0(re(t))}.
\end{equation*}
Hence, using \eqref{eq:euclidean-parity}, \eqref{eq: Formula for multiplier}, and recalling the definitions \eqref{eq: defn of I_{d,n} } of $I_{d,n}$ and \eqref{eq:euclidean-Phi} of $\Phi_{d,n},$ we get
\begin{equation}
\label{eq:euclidean-Cauchy-quotient}
\begin{split}
 \varepsilon^n m_n^d\left(\eta+\frac{1-\varepsilon}{2}\mathbf h\right)
 &=
 \frac{\displaystyle
  \varepsilon^n\int_{-1/2}^{1/2}
 \frac{\prod_{j=1}^d H_{\eta_j+(1-\varepsilon)/4}(re(t))}{1-re(t)} e(-nt)\dd t}
 {r^n|\mathcal{B}_{n}^d|} \\
 &=\frac{\displaystyle
  \varepsilon^n\int_{-1/2}^{1/2}
 \frac{\prod_{j=1}^d H_{\eta_j}(\varepsilon re(t))}{1-re(t)} e(-nt)\dd t}
 {I_{d,n} \Theta(r)^d \rho^{-1}} \\
 &= \frac{\displaystyle
  \int_{-1/2}^{1/2}
 \Phi_{d,n}(t) R_{\varepsilon,\eta}(t) \dd t}
 {I_{d,n}}.
\end{split}
\end{equation}
In the passage from the second to the third line above we changed variable $t \mapsto t+\frac{1-\varepsilon}{4}$. 

In our analysis we shall need the quantity
\[
 \Lambda(\eta)
 =
 \sum_{j=1}^d\sin^2(\pi\eta_j),
\]
which has already played an important role in \cite{BMSWballs}, \cite{NiksinskiWrobel}.
For further reference note that 
\begin{equation*}
%\label{eq:euclidean-Lambda-shift}
 \Lambda(\eta-\mathbf h)
 =\sum_{j=1}^d \cos^2(\pi\eta_j)=
 d-\Lambda(\eta).
\end{equation*}
Thus for every \(\eta \in \T^d\), for at least one of the variables \(\eta\) or
\(\eta-\mathbf h\) the corresponding value of \(\Lambda\) does not exceed $d/2$.  

We now move towards analyzing the multiplier $m_n^d$ for which we shall use the representation
\eqref{eq:euclidean-Cauchy-quotient}.
Choose $c_{\mathrm{loc}}>0$ a sufficiently small absolute constant.
For \(|t|\leq 2c_{\mathrm{loc}}\rho\), the factor
\[
 \frac{1-re(t)}{1-\varepsilon re(t)}.
\]
equals \(1\) when \(\varepsilon=1\), while for \(\varepsilon=-1\)
it is uniformly bounded together with all of its derivatives.  Hence
Corollary \ref{corllary:euclidean-product-derivatives} together with Leibniz rule gives, for every \(m\geq0\),
\begin{equation*}
%\label{eq:euclidean-R-derivatives}
 |\partial_t^mR_{\varepsilon,\eta}(t)|
 \leq
 C_m\rho^{-m}
 (1+U_r(\eta))^m e^{-c_mU_r(\eta)}.
\end{equation*}

We also need a tail estimate. Namely, we shall show that if \(\Lambda(\eta)\leq d/2\), then
\begin{equation}\label{eq:euclidean-numerator-tail}
 |\Phi_{d,n}(t)R_{\varepsilon,\eta}(t)|
 \leq e^{-c\kappa},
 \qquad
\text{for} \  c_{\mathrm{loc}}\rho\leq |t|\leq 1/2.
\end{equation}
Indeed, from definitions we have
\[
 \Phi_{d,n}(t)R_{\varepsilon,\eta}(t)
 =
 e(-nt)\frac{\rho}{1-\varepsilon re(t)}
 \prod_{j=1}^d
 \frac{H_{\eta_j}(re(t))}{H_0(r)},
\]
where
\[
 \left|\frac{\rho}{1-\varepsilon re(t)}\right|\leq1.
\]
Since
\(\Lambda(\eta)\leq d/2\), the set
\[
 J:=\{j:\sin^2(\pi\eta_j)\leq3/4\}
\]
has at least \(d/3\) elements. We will now use Corollary \ref{corollary: psi decay}, thus we recall the definition of $\Delta_r$ used there. Note that  for \(j\in J\) we have $\|\eta_j -1/2\|_{\T} \gtrsim 1$ while $\|t \|_{\T} \ge c\rho$, which implies that for such $j$'s we have
\[
\Delta_r(t,\eta_j) \ge c'.
\]
Hence Corollary \ref{corollary: psi decay} gives us
\[
 \prod_{j=1}^d
 |\psi_r(t,\eta_j)|
 \leq\prod_{j\in J}
 |\psi_r(t,\eta_j)| \le e^{-c''rd}.
\]
Since \(rd\asymp\kappa\), this proves
\eqref{eq:euclidean-numerator-tail}.

Let now $\chi=\one_{[-c_{\mathrm{loc}},c_{\mathrm{loc}}]}$ and set
\[
 \chi_\rho(t)=\chi(t/\rho).
\]
For \(0\leq k\leq2N-1\) we define
\[
 b_{k,d,n}
 =
 \frac1{k!\rho^kI_{d,n}}
 \int_{-1/2}^{1/2}
 \chi_\rho(t)t^k\Phi_{d,n}(t)\dd t.
\]
By Lemma \ref{lemma:euclidean-coefficient-estimate} and
\eqref{eq:euclidean-Phi-local-decay},
\begin{align}
 |b_{k,d,n}|
 &\lesssim
 \frac{\sqrt V}{\rho^k}
 \int_{\mathbb R}|t|^ke^{-cVt^2}\dd t \notag\\
 &\lesssim_k (\rho^2V)^{-k/2}
 \asymp\kappa^{-k/2}.
 \label{eq:euclidean-b-bound}
\end{align}
Next we set
\begin{equation}
     q_{k,r}^{\varepsilon}(\eta)
 =
 \rho^k
 \left.
 \frac{d^k}{d t ^k}R_{\varepsilon,\eta}(t)
 \right|_{t=0}. \label{eq:definition of q_{k,r}}
\end{equation}
Now we shall show that
\begin{equation}\label{eq:euclidean-one-branch-expansion}
 \left|
 \varepsilon^n
 m_n^d\left(\eta+\frac{1-\varepsilon}{2}\mathbf h\right)
 -
 \sum_{k=0}^{2N-1}
 b_{k,d,n}q_{k,r}^{\varepsilon}(\eta)
 \right|
 \lesssim_N \kappa^{-N}
\end{equation}
whenever \(\Lambda(\eta)\leq d/2\).
Indeed, applying Taylor's theorem around \(t=0\) to the expression in
\eqref{eq:euclidean-Cauchy-quotient}, we obtain
\begin{align*} &\left|
 \varepsilon^n
 m_n^d\left(\eta+\frac{1-\varepsilon}{2}\mathbf h\right)
 -
 \sum_{k=0}^{2N-1}
 b_{k,d,n}q_{k,r}^{\varepsilon}(\eta)
 \right| \\
 &\le\frac1{I_{d,n}}\int_{-1/2}^{1/2}\chi_\rho(t)|\Phi_{d,n}(t)|
 \left|R_{\varepsilon,\eta}(t)
 -\sum_{k=0}^{2N-1}\frac{t^k}{k!}R_{\varepsilon,\eta}^{(k)}(0)\right|\dd t\\
 &\quad+\frac1{I_{d,n}}\int_{-1/2}^{1/2}(1-\chi_\rho(t))
 |\Phi_{d,n}(t)||R_{\varepsilon,\eta}(t)|\dd t \\
 &\lesssim_N \sqrt{V} \int_{\R} |t|^{2N} e^{-cVt^2} \rho^{-2N} (1+U_r(\eta))^{2N} e^{-cU_r(\eta)} \dd t + \sqrt{V} e^{-c \kappa} \\
 &\lesssim_N V^{-N} \rho^{-2N}+ \sqrt{V} e^{-c \kappa} \lesssim_N \kappa^{-N},
\end{align*}
In the above chain of inequalities, estimating the first integral we used Taylor's theorem together with
Corollary~\ref{corllary:euclidean-product-derivatives} and
\eqref{eq:euclidean-Phi-local-decay}, while estimating the second integral we used
tail estimate \eqref{eq:euclidean-numerator-tail}.
Finally, in the last inequality we used
\eqref{eq: bound on V in terms of power of kappa}. Hence,  \eqref{eq:euclidean-one-branch-expansion} is proved.
Note that Corollary \ref{corllary:euclidean-product-derivatives} gives
\[
 |q_{k,r}^{\varepsilon}(\eta)|
 \leq
 C_k(1+U_r(\eta))^ke^{-c_kU_r(\eta)}.
\]
Hence if \(\Lambda(\eta)\geq d/2\), then  by \eqref{eq: kappa, r , rho properties}
\[
U_r(\eta)
 \geq \frac{rd}{2\rho}
 \gtrsim\kappa.
\]
Consequently, for every \(N\geq1\),
\begin{equation}\label{eq:euclidean-q-rapid}
 |q_{k,r}^{\varepsilon}(\eta)|
 \lesssim_{k,N}
 U_r(\eta)^{-N}\lesssim_{k,N}\kappa^{-N}.
\end{equation}
Define now
\begin{equation}\label{eq:euclidean-a-approximation}
 a_{n,N}^d(\xi)
 =
 \sum_{k=0}^{2N-1}
 b_{k,d,n}
 \left(
 q_{k,r}^{+}(\xi)
 +(-1)^nq_{k,r}^{-}(\xi-\mathbf h)
 \right).
\end{equation}
For every \(\xi \in \T^d\), either
\(\Lambda(\xi)\leq d/2\) or
\(\Lambda(\xi-\mathbf h)\leq d/2\) holds.
In the first case, we apply
\eqref{eq:euclidean-one-branch-expansion} with
\(\varepsilon=1\) and \(\eta=\xi\). Note that the sum of the terms with arguments $\xi-\mathbf h$ is then $O(\kappa^{-N})$ by
\eqref{eq:euclidean-b-bound} and
\eqref{eq:euclidean-q-rapid}.  In the second case we apply
\eqref{eq:euclidean-one-branch-expansion} with
\(\varepsilon=-1\) and \(\eta=\xi-\mathbf h\). Here the sum of the terms with argument $\xi$ is
treated in an analogous way. In summary, we have justified that
\begin{equation}\label{eq:euclidean-arbitrary-order}
 \sup_{\xi\in\mathbb T^d}
 |m_n^d(\xi)-a_{n,N}^d(\xi)|
 \lesssim_N \kappa^{-N}
 \lesssim_N\min\{n,d\}^{-N}.
\end{equation}

To finish the proof, it remains to express the functions
\(q_{k,r}^{\varepsilon}\) defined in \eqref{eq:definition of q_{k,r}} as finite linear combinations of normalized
discrete Gaussian multipliers. The purpose of the next argument is to express the functions \(q_{k,r}^{\varepsilon}\) in terms of derivatives of the normalized discrete Gaussian multipliers with respect to the parameter $r$.  We recall that these multipliers are defined by
\[
 \widehat g_r^d(\eta)
 =
 \prod_{j=1}^d
 \frac{H_{\eta_j}(r)}{H_0(r)}.
\]
We introduce the parameter
\[
 s=s(r)=\log\frac r{1-r},\qquad \textrm{ so that}
 \qquad
 r=r(s)=\frac{e^s}{1+e^s},
\]
and let
\[
 \partial_s=\rho r\partial_r.
\]
Denote
\[
 \mathcal F_{\varepsilon,\eta}(z)
 =
 \frac{1-z}{1-\varepsilon z}
 \prod_{j=1}^d
 \frac{H_{\eta_j}(z)}{H_0(z)},
\]
and notice that we have
\[
 R_{\varepsilon,\eta}(t)
 =
 \mathcal F_{\varepsilon,\eta}(re(t)).
\]
Writing \(D=r\partial_r\), by the definition \eqref{eq:definition of q_{k,r}} we therefore obtain
\[
 q_{k,r}^{\varepsilon}(\eta)
 =
 (2\pi i)^k\rho^kD^k
 \left[
 a_\varepsilon(r)\widehat g_r^d(\eta)
 \right],
\]
where
\[
 a_+(r)=1,
 \qquad
 a_-(r)=\frac{1-r}{1+r}.
\]
Now, the equality \(\partial_s=\rho D\) implies the identity
\begin{equation*}
%\label{eq:euclidean-radial-recurrence}
 \rho^{k+1}D^{k+1}F
 =
 (\partial_s+kr)(\rho^kD^kF).
\end{equation*}
This way we obtain the following recurrence formula
\[
 q_{k+1,r}^{\varepsilon}
 =
 2 \pi i(\partial_s+kr)q_{k,r}^{\varepsilon}.
\]
Starting with
\[
 q_{0,r}^{\varepsilon}
 =
 a_\varepsilon(r)\widehat g_r^d,
\]
a simple induction gives
\begin{equation}\label{eq:euclidean-q-s-derivatives}
 q_{k,r}^{\varepsilon}(\eta)
 =
 \sum_{j=0}^k
 c_{k,j}^{\varepsilon}(r)
 \partial_s^j
 \widehat g_{r(s)}^d(\eta),
 \qquad
 |c_{k,j}^{\varepsilon}(r)|\lesssim_k 1.
\end{equation}
Moreover we have a recurrence for the coefficient $c_{k,j}^\varepsilon$ given by
\[
 c_{k+1,j}^{\varepsilon}
 =2 \pi i\left(\partial_sc_{k,j}^{\varepsilon}
 +kr c_{k,j}^{\varepsilon}+c_{k,j-1}^{\varepsilon}\right),
\]
where $c_{k,j}^\varepsilon=0$ for $j \not \in \{0, \ldots, k\}$. A straightforward calculation shows that all \(s\)-derivatives of \(r(s)\) and \(a_-(r(s))\) are bounded. Therefore, an induction
on \(k\), performed simultaneously for every \(m\geq0\) gives
\[
 \sup_{s\in\R}
 |\partial_s^m c_{k,j}^{\varepsilon}(r(s))|\lesssim_{k,m} 1.
\]
Taking \(m=0\) proves the coefficient estimate in
\eqref{eq:euclidean-q-s-derivatives}.
Since
\[
 \partial_s(\rho^jD^jF)
 =-jr\rho^jD^jF+\rho^{j+1}D^{j+1}F,
\]
and \(\partial_s=\rho D\) preserves polynomials in \(r\),
another straightforward induction gives, for $m\geq1$,
\[
 \partial_s^m=\sum_{j=1}^m\lambda_{m,j}(r)\rho^jD^j,
\]
where the coefficients $\lambda_{m,j}$ are polynomials in \(r\) and hence are uniformly bounded for $0<r<1$.  Moreover,
\[
 \rho^jD^j\widehat g_r^d(\eta)
 =(2 \pi i)^{-j}\rho^j\left.\partial_t^j
 \prod_{\nu=1}^d\frac{H_{\eta_\nu}(re(t))}{H_0(re(t))}
 \right|_{t=0},
\]
thus
taking $t=0$ in
Corollary \ref{corllary:euclidean-product-derivatives} gives
\begin{equation}
\sup_{d\geq1}\sup_{s\in\R}\sup_{\eta\in\mathbb T^d}
 |\partial_s^m\widehat g_{r(s)}^d(\eta)|\lesssim_m 1 \label{eq:euclidean-Gaussian-s-derivatives},
\end{equation}
for every $m\geq1$.

The remaining part is to replace these \(s\)-derivatives in \eqref{eq:euclidean-q-s-derivatives} by finite differences. Before proceeding further we will need an  elementary observation.
Let $j\geq0$, $M\geq1$, $h>0,$  and let $I$ be an open interval
containing $[s,s+(j+M-1)h]$. We claim that there are constants
$\gamma_{\ell,j,M}$, $0\leq\ell\leq j+M-1$, depending only on $j,M$,
such that every function $F\in C^{j+M}(I)$, satisfies
\begin{equation}\label{eq:euclidean-finite-difference}
 \begin{aligned}
 F^{(j)}(s)
 &=h^{-j}\sum_{\ell=0}^{j+M-1}\gamma_{\ell,j,M}F(s+\ell h)\\
 &\quad+O_{j,M}\left(
 h^M\sup_{s\leq u\leq s+(j+M-1)h}|F^{(j+M)}(u)|
 \right).
 \end{aligned}
\end{equation}
Indeed, the coefficients are uniquely determined by the following system of linear equations
\[
 \sum_{\ell=0}^{j+M-1}\gamma_{\ell,j,M}\ell^q
 =
 \begin{cases}
  j!,&q=j,\\
  0,&0\leq q\leq j+M-1,\ q\ne j.
 \end{cases}.
\]
The matrix corresponding to this system is a Vandermonde matrix, and hence its determinant is given by
\[
 \prod_{0\leq a<b\leq j+M-1}(b-a)\ne0,
\]
so such coefficients exist. Formula \eqref{eq:euclidean-finite-difference} now follows by applying Taylor's theorem with remainder to each \(F(s+\ell h)\), multiplying the resulting identity by \(h^{-j}\gamma_{\ell,j,M}\), and summing over \(\ell\).

At this point we have established all the ingredients needed to complete the proof of Proposition \ref{prop:euclidean-Gaussian-approximation}. Take
\[
 h=\kappa^{-1/2},
 \qquad
 M=2N.
\]
For $j\leq k$, inequality \eqref{eq:euclidean-b-bound} gives
\[
 |b_{k,d,n}|h^{-j}
 \lesssim_k \kappa^{-k/2}\kappa^{j/2}
 \lesssim_N 1.
\]
The error in
\eqref{eq:euclidean-finite-difference}, after multiplication
by $b_{k,d,n}$, is thus at most
\[
 C_{N}|b_{k,d,n}|h^{2N}
 \lesssim_N\kappa^{-N}.
\]
For $0\leq j\leq k$, we apply
\eqref{eq:euclidean-finite-difference} to
$F(u)=\widehat g_{r(u)}^d(\eta)$ and set
\[
 r_\ell=r(s+\ell h)
 =\frac{e^{s+\ell h}}{1+e^{s+\ell h}},
 \qquad 0\leq\ell\leq j+2N-1.
\]
This way due to
\eqref{eq:euclidean-Gaussian-s-derivatives}, we obtain
\[
 \partial_s^j\widehat g_{r(s)}^d(\eta)
 =
 h^{-j}
 \sum_{\ell=0}^{j+2N-1}
 \gamma_{\ell,j,2N}
 \widehat g_{r(s+\ell h)}^d(\eta)
 +
 O_{j,N}(\kappa^{-N}).
\]
After multiplication by
\(b_{k,d,n}c_{k,j}^{\varepsilon}(r)\), the coefficients in this
finite sum are bounded by \(C_{N}\), while the error is
\(O_{N}(\kappa^{-N})\). Substituting these finite-difference expansions into
\eqref{eq:euclidean-q-s-derivatives} and then into
\eqref{eq:euclidean-a-approximation}, we see that
\begin{equation}
\label{eq: atil to a}
a_{n,N}^d(\xi)=\tilde{a}_{n,N}^d(\xi)+O_{N}(\kappa^{-N}),
\end{equation}
where \(\tilde{a}_{n,N}^d\) is a finite sum of the form
\[
 \tilde{a}_{n,N}^d(\xi)
 =
 \sum_{\nu=1}^{L_N}
 \alpha_{\nu,d,n}
 \widehat g_{r_{\nu,d,n}}^d(\xi)
 +
 \sum_{\nu=1}^{L_N}
 \beta_{\nu,d,n}
 \widehat g_{\widetilde r_{\nu,d,n}}^d(\xi-\mathbf h),
\]
with \(L_N\) depending only on \(N\), and
\[
 \sum_{\nu=1}^{L_N}
 \bigl(
 |\alpha_{\nu,d,n}|
 +
 |\beta_{\nu,d,n}|
 \bigr)
 \lesssim_N 1.
\]
Finally, combining \eqref{eq: atil to a} with
\eqref{eq:euclidean-arbitrary-order} and
\(\kappa\asymp\min\{n,d\}\), we obtain
\[
 \sup_{\xi\in\mathbb T^d}
 |m_n^d(\xi)-\tilde{a}_{n,N}^d(\xi)|
 \lesssim_N\kappa^{-N}
 \lesssim_N\min\{n,d\}^{-N},
\]
concluding the proof of Proposition
\ref{prop:euclidean-Gaussian-approximation}.
\end{proof}

\section{Discrete Euclidean spheres}
\label{sec: spheres}
In this section we point out how to use the techniques from previous section to prove Proposition \ref{prop:sphere-Gaussian-approximation} and then conclude Theorem \ref{thm:main:sphere}. Since the scheme of the proof is essentially the same, we keep the discussion brief. Throughout this section, $d\geq5$ and $1\leq n\leq d^4$.

The generating function formula takes the following form
\begin{align*}
    \sum_{n=0}^\infty |\mathcal{S}_n^{d}| s_n^d(\xi) z^n&= \sum_{n=0}^\infty \sum_{\substack{x \in \Z^d,\\|x|^2 = n}}  z^n e(\xi \cdot x)\\ 
    &=  \sum_{\substack{x \in \Z^d}} e(\xi \cdot x) z^{|x|^2} =\prod_{j=1}^d H_{\xi_j}(z)
\end{align*}
and hence by orthogonality for any $r \in (0,1)$ we have 
\begin{equation*} 
%\label{eq: Formula for spherical multiplier}
   r^n |\mathcal{S}_n^{d}| s_n^d(\xi)= \int_\T \prod_{j=1}^d H_{\xi_j}(re(t)) e(-nt) \dd t,
\end{equation*}
in particular 
\begin{equation*}
   r^n |\mathcal{S}_n^{d}| = \int_\T \Theta(re(t))^d e(-nt) \dd t .
\end{equation*}
This time we define $r=r_{{\textrm{sph}},n,d}$ and $V=V_{{\textrm{sph}}}$ by the formulas
\[
\mu(r)= \frac{n}{d}, \qquad V=d r \mu'(r), \qquad \rho=1-r, \qquad \kappa=V \rho^2. 
\]
The second part of Lemma \ref{lem: basic aproximation for mu} gives us the same approximations as before
 \[
n \asymp \frac{dr}{\rho}, \qquad V \asymp \frac{dr}{\rho^2},
 \]
and
 \begin{equation*} 
     r \asymp \frac{n}{n+d}, \qquad \rho \asymp \frac{d}{n+d}, \qquad \kappa=\rho^2 V \asymp \min\{n,d \}.
 \end{equation*}
Defining
\begin{equation*}
%\label{eq:spherical-Phi}
\Phi^{\textrm{sph}}_{d,n}(t)
 =e(-nt)
 \left(\frac{H_0(re(t))}{H_0(r)}\right)^d,
\end{equation*}
and using orthogonality we have
\[
\int_{-1/2}^{1/2} \Phi^{\textrm{sph}}_{d,n}(t) \dd t =\Theta(r)^{-d} r^n |\mathcal{S}_n^d|.
\]

The required analogue of Lemma \ref{lemma:decay of Phi} is stated below. Its proof is the same, except for the fact that now there is no need to bound $|\frac{\rho}{1-re(t)}|$.
\begin{lemma} \label{lemma:decay of spher Phi}
For every constant $c_1 \in (0,1/2)$ there exists $c_2>0$ such that we have  
\begin{equation*}
 |\Phi^{\textrm{sph}}_{d,n}(t)|\leq e^{-c_2Vt^2},
 \qquad \text{for} \ |t|\leq c_1\rho.
\end{equation*}
and
\begin{equation*}
 |\Phi^{\textrm{sph}}_{d,n}(t)|\leq e^{-c_2\kappa}, \qquad \text{for} \ c_1 \rho \le |t| \le 1/2
\end{equation*}
\end{lemma}
We also need an analogue of Lemma \ref{lemma:euclidean-coefficient-estimate}, which we state below. It's proof is again exactly the same with same bounds except for the fact that we don't need to prove the estimate \eqref{eq:euclidean-geometric-third-derivative}.
\begin{lemma}
  For every $d\geq5$ and every integer $1\leq n\leq d^4$, we have
    \[
    \int_{-1/2}^{1/2} \Phi^{\textrm{sph}}_{d,n}(t) \dd t  \asymp V^{-1/2}.
    \]
\end{lemma}

Now, for any $\varepsilon \in \{-1,1 \}$ we define the analogue of $ R_{\varepsilon,\eta}$ by
\[
 R_{\varepsilon,\eta}^{\textrm{sph}}(t)
 =
 \prod_{j=1}^d
 \frac{H_{\eta_j}(re(t))}{H_0(re(t))}.
\]
Then we have the following analogue of \eqref{eq:euclidean-Cauchy-quotient}
\[
\varepsilon^n s_n^d\left(\eta + \frac{1-\varepsilon}{2} \mathbf{h} \right)= \frac{\displaystyle
  \int_{-1/2}^{1/2}
 \Phi^{\textrm{sph}}_{d,n}(t) R_{\varepsilon,\eta}^{\textrm{sph}}(t) \dd t}
 {\int_{-1/2}^{1/2}
 \Phi^{\textrm{sph}}_{d,n}(t) \dd t}.
\]
At this point, to complete the proof of Proposition \ref{prop:sphere-Gaussian-approximation}, we proceed exactly as in the proof of Proposition \ref{prop:euclidean-Gaussian-approximation}.

Finally, to prove Theorem \ref{thm:main:sphere} it suffices to consider \(p=2\) (by interpolation) and sufficiently large \(d\). Let \(\mathcal T_{n,4}^d\) be the approximating operators provided by Proposition \ref{prop:sphere-Gaussian-approximation} with \(N=4\). As in Section \ref{ssec: main from prop}, their maximal function is bounded on \(\ell^2(\mathbb Z^d)\) by Theorem \ref{thm:gaussian}. Moreover, Proposition \ref{prop:sphere-Gaussian-approximation} and Parseval’s identity give
\begin{align*}
\left\|\sup_{1\le n\le d^4}
|(\mathcal A_n^d-\mathcal T_{n,4}^d)f|\right\|_2
&\le
\sum_{1\le n\le d^4}
\|(\mathcal A_n^d-\mathcal T_{n,4}^d)f\|_2\\
&\lesssim
\sum_{1\le n\le d^4}\min\{n,d\}^{-4}
\|f\|_2
\lesssim \|f\|_2.
\end{align*}
Combining this estimate with \eqref{eq:max sphere large scale l2} concludes the proof of Theorem \ref{thm:main:sphere}.

\microtypesetup{expansion=false}

\end{document}